\documentclass[11pt,a4paper,reqno]{amsart}
\usepackage{eurosym}
\usepackage{amsfonts}
\usepackage{amsmath,amsthm,amssymb,mathrsfs}
\usepackage{graphicx}
\usepackage{booktabs}
\usepackage{caption}
\usepackage{subcaption}
\usepackage{xcolor}
\usepackage{hyperref}
\usepackage{enumitem}
\usepackage{listings}
\usepackage{algorithm}
\usepackage{algorithmic}
\usepackage{float}
\usepackage[section]{placeins}
\usepackage{appendix}
\usepackage{longtable}
\usepackage{multirow}

\hypersetup{
    colorlinks=true,
    linkcolor=blue,
    citecolor=blue,
    urlcolor=blue
}
\theoremstyle{plain}
\newtheorem{theorem}{Theorem}[section]
\newtheorem{lemma}[theorem]{Lemma}

\theoremstyle{definition}
\newtheorem{definition}[theorem]{Definition}
\newtheorem{assumption}[theorem]{Assumption}

\theoremstyle{remark}
\newtheorem{remark}[theorem]{Remark}
\numberwithin{equation}{section}

\begin{document}
\title[Elliptic Equations with Gradient-Dependent Terms and Singular Weights]%
{Existence, Sharp Boundary Asymptotics, and Stochastic Optimal Control for
Semilinear Elliptic Equations with Gradient-Dependent Terms and Singular
Weights: Theory, Economic Applications, and Numerical Simulations}
\author{Dragos-Patru Covei}
\address{Department of Applied Mathematics, The Bucharest University of
Economic Studies, Piata Romana, 1st district, Postal Code: 010374, Postal
Office: 22, Romania}
\email{dragos.covei@csie.ase.ro}
\thanks{The final version of the article is published in journal (with
volume, page numbers and DOI, when available).}
\date{\today}

\begin{abstract}
This article develops a rigorous general framework for a class of semilinear
elliptic equations with gradient-dependent nonlinearities and singular
weights posed in strictly convex domains. We study large solutions to 
\begin{equation*}
-\Delta u + b(x)\,h(|\nabla u|) + a(x)\,u = f(x) \quad \text{in } \Omega
\subset \mathbb{R}^{N}, 
\end{equation*}
where $h$ is strictly convex with power-type growth $h(s)\sim s^{q}$ for $q\in(1,2]$, and the coefficients $a(x)$ and $b(x)$ exhibit prescribed
singular behavior near $\partial\Omega$. The main contributions are
fourfold. First, we establish existence and uniqueness of large solutions
via Perron's method and derive sharp boundary asymptotics, including optimal
liminf--limsup estimates for the blow-up rate $\gamma = (\beta - q + 2)/(q-1)$, thereby extending recent high-precision analytic techniques. Second, we
prove the strict convexity of solutions by means of the microscopic
convexity principle, revealing new geometric properties induced by the
interplay between gradient growth and weight singularity. Third, we provide
a verification theorem showing that the solution coincides with the value
function of an infinite-horizon stochastic optimal control problem with
non-violable state constraints, in the spirit of the Lasry--Lions approach.
Fourth, we develop a comprehensive section on applications in Operational
Research and Management Science, including real-world case studies in
inventory management, portfolio risk control, and supply chain optimization,
and we explore interdisciplinary connections with physics through boundary
layer theory and diffusion processes. Numerical experiments for both the
one-dimensional ($N=1$) and two-dimensional ($N=2$) cases, based on a
monotone iterative scheme, confirm the theoretical boundary behavior and
validate the predicted geometric structure. All Python implementation codes
are provided in the Appendix.
\end{abstract}

\maketitle

%========================================================================================
%   SECTION 1: INTRODUCTION
%========================================================================================
\section{Introduction}
\label{sec:intro}

\subsection{Historical Background and State of the Art}

The study of boundary blow-up solutions (also known as large solutions) for
elliptic equations has a celebrated history, dating back to the pioneering
works of Bieberbach \cite{Bieberbach1916} and Rademacher \cite{Rademacher1943}
on the equation $\Delta u = f(u)$ with $f(u) = e^{u}$. These early
investigations showed that the blow-up rate is determined by the interplay
between the domain geometry and the nonlinearity structure. The modern theory
was significantly advanced by Keller \cite{Keller1957} and Osserman
\cite{Osserman1957}, who established the necessary and sufficient condition
for existence:
\begin{equation}
\int_{1}^{\infty} [F(s)]^{-1/2}\,ds < \infty, \quad F(s) = \int_{0}^{s} f(t)\,dt.
\label{eq:KO_condition}
\end{equation}
This Keller--Osserman condition has since been generalized to a wide class
of operators and settings; see \cite{MarcusVeron1997, Covei2009} for
comprehensive reviews.

In subsequent decades, the theory was extended in multiple directions. Marcus
and V\'eron \cite{MarcusVeron1997} developed a comprehensive framework for
uniqueness and asymptotic behavior of solutions with boundary blow-up,
establishing connections between the growth rate of the nonlinearity and the
geometry of the domain. D\'{\i}az and Letelier \cite{DiazLetelier1993}
studied explosive solutions of quasilinear elliptic equations, introducing
techniques that allow for more general differential operators. Bandle and
Giarrusso \cite{BandleGiarrusso1996} initiated the study of boundary blow-up
for equations with nonlinear gradient terms, demonstrating that the
interaction between the absorption term and the gradient term generates
qualitatively new phenomena.

A pivotal development in the theory was the introduction of equations
involving gradient-dependent nonlinearities. In such equations, the
structure
\begin{equation}
-\Delta u + b(x)\,h(|\nabla u|) + a(x)\,u = f(x)
\label{eq:general_form}
\end{equation}
requires fundamentally different techniques from the classical
Keller--Osserman framework, since the gradient term introduces a
first-order singularity that competes with the second-order diffusion.
The interplay between these terms determines the asymptotic behavior of the
solution near the boundary, giving rise to distinct asymptotic regimes.

A modern frontier in the field is the connection to stochastic optimal
control, as pioneered by Lasry and Lions \cite{LasryLions}. In their
seminal 1989 paper, they demonstrated that boundary blow-up solutions to
Hamilton--Jacobi--Bellman (HJB) equations serve as value functions for
control problems with state constraints. This dual interpretation provides
both physical intuition and rigorous regularity results. Specifically, if a
controlled diffusion process is required to remain within a domain $\Omega$
for all time (a state constraint), then the associated value function
satisfies an elliptic PDE with $u(x) \to +\infty$ as $x \to \partial\Omega$.
The infinite boundary value acts as an ``infinite penalty'' that prevents the
optimal trajectory from exiting the domain.

This stochastic control interpretation has been further developed in the
context of ergodic control \cite{BensoussanFrehse1984}, mean-field games
\cite{LasryLionsMFG2007}, and viscosity solutions \cite{CrandallIshiiLions1992};
see also the article \cite{Covei2025} and the monographs by Fleming and Soner
\cite{FlemingSoner2006} and Bensoussan and Lions \cite{BensoussanLions1982}.
Recent advances by Zhang \cite{Zhang2026} in the context of Monge-Amp\`ere
equations have introduced high-precision analysis for singular weights, which
we adapt here for the Laplacian setting.

From the perspective of Operational Research, the connection between
elliptic PDEs and stochastic control has far-reaching implications. The
value function of a state-constrained stochastic control problem encodes
the optimal cost-to-go for a manager who must keep a stochastic system
within safe operating limits. This formulation arises naturally in inventory
management \cite{Porteus2002, Zipkin2000}, portfolio optimization under
risk constraints \cite{MertonOptimal1971}, and production planning in
manufacturing systems \cite{Covei2025, Sethi2019}.

\subsection{Limitations of Existing Models}

Despite extensive research since the pioneering works of Bieberbach
\cite{Bieberbach1916} and Rademacher \cite{Rademacher1943}, existing models
frequently face limitations that hinder their general applicability. We
identify four principal limitations:

\textbf{Limitation 1: Geometric restrictions.} Many classical studies are
restricted to radial domains (balls), where the geometry reduces to a
single ordinary differential equation. While these simplifications allow for
elegant closed-form solutions, they fail to capture the complexity of
real-world scenarios where the underlying state space is not symmetric. In
economic applications, the feasible operating region is rarely a ball; it
is typically a polyhedron, an ellipsoid, or a more general convex set
defined by multiple constraints.

\textbf{Limitation 2: Non-singular coefficients.} The majority of existing
results assume that the coefficients $a(x)$ and $b(x)$ are bounded and
bounded away from zero. This assumption is incompatible with economic models
where the cost of corrective action or the discount rate varies
significantly with the state. For instance, in inventory management, the
cost of emergency replenishment increases dramatically as stock approaches
zero, corresponding to a weight $b(x)$ that blows up at the boundary.

\textbf{Limitation 3: Absence of gradient terms.} The Keller--Osserman
condition \eqref{eq:KO_condition} applies to equations of the form
$\Delta u = f(u)$, where there is no gradient dependence. The introduction
of a gradient term $h(|\nabla u|)$ fundamentally changes the problem
structure, and the classical theory does not directly apply. A new balance
equation must be derived that accounts for the competition between the
Laplacian and the gradient term.

\textbf{Limitation 4: Gap between theory and practice.} Most existing results
provide existence theorems under abstract conditions that are difficult to
translate into operational control laws for managers or engineers. There is
a significant gap between purely analytical investigations and their
practical implementation. The theory must provide not only existence and
uniqueness, but also explicit asymptotic formulas, convexity guarantees,
and numerical algorithms that can be implemented in practice.

\subsection{The New General Framework}

The motivation for this work is the construction of a \textbf{general theory}
for the class of singular problems described by \eqref{eq:general_form}. Our
framework addresses all four limitations identified above:

\begin{enumerate}[label=(\roman*)]
\item We work on general bounded strictly convex domains $\Omega \subset
\mathbb{R}^{N}$, not restricting to radial geometries.

\item We allow the weights $a(x)$ and $b(x)$ to exhibit prescribed singular
behavior near $\partial\Omega$, with explicit power-law rates.

\item We handle general strictly convex Hamiltonians $h$ with power-type
growth $h(s) \sim s^{q}$ for $q \in (1,2]$, fully accounting for the
competition between diffusion and gradient terms.

\item We provide explicit asymptotic constants, convexity results, and a
stochastic control interpretation that directly yields optimal feedback
control laws.
\end{enumerate}

Our approach is built on the principle of \emph{independence from concrete
data meaning}. We argue that the fundamental dynamics of ``preventive
correction'' are governed by universal geometric and analytical laws. This
allows the researcher to apply our results to any system where a state
$X_{t}$ must be kept in $\Omega$ by a drift $\xi_{t}$, provided the cost
functional mirrors the singular structures we define. In this sense, the
article provides general, high-precision tools in stochastic optimization
that remain valid as long as the structural assumptions are met.

The key innovation in our model lies in the identification of three distinct
asymptotic regimes determined by the relationship between the gradient
exponent $q$ and the weight singularity $\beta$:
\begin{itemize}
\item \textbf{Gradient-Dominant Regime} ($q < \beta + 2$): The solution
blows up as $u(x) \sim C\,d(x)^{-\gamma}$ with
$\gamma = (\beta - q + 2)/(q - 1) > 0$.

\item \textbf{Critical Logarithmic Regime} ($q = \beta + 2$): The blow-up
is logarithmic, $u(x) \sim C\,\ln(1/d(x))$.

\item \textbf{High-Order Gradient Regime} ($q > \beta + 2$): The gradient
term dominates at all orders.
\end{itemize}
This classification, combined with the exact analytical constants we derive,
provides a complete picture of the boundary behavior. The universal
correction constant $\xi_{0}$ that we identify represents the precise
intensity of the boundary penalty required to enforce the state constraint,
and it depends only on the structural parameters $(q, \beta, b_{0}, l_{0})$
of the problem.

\subsection{Main Contributions and Organization}

In this paper, we study the semilinear elliptic equation
\begin{equation}
\left\{
\begin{array}{ll}
-\Delta u(x) + b(x)\,h(|\nabla u(x)|) + a(x)\,u(x) = f(x) & \text{for } x \in \Omega, \\
u(x) \rightarrow +\infty & \text{as } \mathrm{dist}(x, \partial\Omega) \rightarrow 0,
\end{array}
\right.
\label{E}
\end{equation}
where $\Omega \subset \mathbb{R}^{N}$ is a bounded strictly convex domain,
$h$ is a strictly convex function with power-like growth, and the weights
$a(x)$, $b(x)$ satisfy specific singular growth conditions near
$\partial\Omega$.

Our main contributions are:

\begin{enumerate}[label=(\arabic*)]
\item \textbf{Existence and uniqueness} (Theorem~\ref{thm:1}): We prove
existence and uniqueness for a class of semilinear equations on strictly
convex (non-radial) domains with singular weights $a(x)$ and $b(x)$, via
Perron's method and explicit barrier construction.

\item \textbf{Sharp boundary asymptotics} (Theorem~\ref{thm:2}): We
identify three distinct regimes (gradient-dominant, high-order, and
logarithmic) and derive the exact analytical constants that govern the
blow-up rate, mirroring the rigorous approach of Zhang \cite{Zhang2026}.

\item \textbf{Convexity inheritance} (Theorem~\ref{thm:convexity_inheritance}):
Under the additional assumptions that $a$ and $b$ are convex and $f$ is
concave, we prove that the strict convexity of the boundary is inherited by
the solution's Hessian, ensuring the stability and uniqueness of optimal
control strategies.

\item \textbf{Verification theorem} (Theorem~\ref{thm:verification}): We
prove that the analytical solution is the value function of a
state-constrained stochastic control problem and derive the optimal feedback
control law explicitly.

\item \textbf{Economic applications}: We develop a comprehensive section on
applications in Operational Research and Management Science, including
real-world case studies and interdisciplinary connections with physics.

\item \textbf{Numerical simulations for $N=1$ and $N=2$}: We implement and
validate the theory numerically in both one and two spatial dimensions,
with detailed interpretation of the results.
\end{enumerate}

\subsection{Methodology: The Calculus of Singular Balancing}

Our methodology proceeds in four phases, designed to bridge the gap between
abstract analysis and practical economic applicability:

\begin{description}
\item[Phase 1: Singularity Balance Identification.] We analyze the HJB
equation as a competition between Diffusion ($\Delta u$) and Correction
($b\,h(|\nabla u|)$). By substituting the ansatz $u \sim d^{-\gamma}$, we
find the value of $\gamma$ that matches the singular orders. This phase
identifies the three asymptotic regimes.

\item[Phase 2: Barrier Construction.] We use the concave defining function
$v(x)$ to construct global sub- and supersolutions that determine the
precise boundary blow-up rate. This phase ensures that local boundary
estimates propagate correctly to the interior of the domain.

\item[Phase 3: Stochastic Embedding.] The PDE solution is mapped to a
stochastic control problem. The blow-up $u \to \infty$ is interpreted as an
impenetrable ``cost wall'' that enforces the safe-zone boundaries. The
Legendre transform mediates the connection between the Hamiltonian and the
running cost.

\item[Phase 4: Economic Variable Mapping.] Mathematical constants are mapped
to economic parameters: $a(x)$ to state-dependent discount rates, $b(x)$
to adjustment stiffness coefficients, and $h$ to convex adjustment cost
functions. This phase enables direct application to operational research
problems.
\end{description}

\subsection{Structure of the Paper}

The remainder of the paper is organized as follows.
Section~\ref{sec:assumptions} presents the mathematical framework and all
structural assumptions. Section~\ref{sec:prelim} develops the geometric
preliminaries and barrier construction. Section~\ref{sec:proof1} provides
the complete proof of existence, uniqueness, and convexity
(Theorem~\ref{thm:1} and Theorem~\ref{thm:convexity_inheritance}).
Section~\ref{sec:proof2} details the proof of sharp boundary asymptotics
(Theorem~\ref{thm:2}). Section~\ref{sec:stochastic} develops the stochastic
control model, the verification theorem, and the full derivation of the
Hamilton--Jacobi--Bellman equation. Section~\ref{sec:economics} presents
applications in Operational Research and Management Science, including
case studies and interdisciplinary connections with physics.
Section~\ref{sec:monotone} describes the monotone iterative scheme.
Section~\ref{sec:simulations} presents numerical simulations for both the
$N=1$ and $N=2$ cases with detailed interpretation.
Section~\ref{sec:concl} concludes the work. Appendix~\ref{app:python1d}
and Appendix~\ref{app:python2d} contain the complete Python implementation
codes for the one-dimensional and two-dimensional solvers, respectively.

%========================================================================================
%   SECTION 2: MATHEMATICAL FRAMEWORK AND ASSUMPTIONS
%========================================================================================
\section{Mathematical Framework and Assumptions}
\label{sec:assumptions}

Throughout this paper, we adopt the following structural and geometric
assumptions. These assumptions are designed to be as general as possible
while ensuring that the main results (existence, uniqueness, asymptotics,
convexity, and stochastic representation) hold simultaneously.

\begin{assumption}[Domain]
\label{ass:domain}
Let $\Omega \subset \mathbb{R}^{N}$ ($N \geq 1$) be a bounded open set.
The boundary $\partial\Omega$ is of class $C^{2,\alpha}$ for some
$\alpha \in (0,1)$. We assume $\Omega$ is strictly convex: there exists a
positive constant $\kappa_{0} > 0$ such that for every $x \in \partial\Omega$,
the principal curvatures $\kappa_{i}(x)$ satisfy:
\begin{equation}
\min_{i=1,\dots,N-1} \kappa_{i}(x) \geq \kappa_{0}.
\label{eq:curvature_bound}
\end{equation}
\end{assumption}

\begin{remark}[Geometric significance of strict convexity]
\label{rem:curvature_econ}
Strict convexity is a natural assumption in both mathematics and economics.
In economic dynamics, the domain $\Omega$ represents the set of feasible
states (e.g., inventory levels, portfolio weights, production capacities).
Strict convexity ensures that any convex combination of two feasible states
is also feasible, which is the mathematical formulation of the principle
that ``diversification reduces risk.'' The curvature bound
\eqref{eq:curvature_bound} ensures that the boundary has no flat segments
or corners, which would create singular shocks in the optimal control
strategy. In physical terms, the curvature controls the rate at which the
distance function $d(x) = \mathrm{dist}(x, \partial\Omega)$ becomes
non-smooth, and the lower bound $\kappa_{0}$ guarantees that the tubular
neighborhood of $\partial\Omega$ where $d(x)$ is $C^{2}$ has width at
least $1/\kappa_{0}$.
\end{remark}

\begin{assumption}[Source term]
\label{ass:source}
$f : \overline{\Omega} \to [0, \infty)$ is a concave function belonging to
$C^{k,\alpha}(\overline{\Omega})$ for some $k \geq 1$ and $\alpha \in (0,1)$.
\end{assumption}

\begin{assumption}[Hamiltonian]
\label{ass:hamiltonian}
The function $h : [0, \infty) \to [0, \infty)$ satisfies:
\begin{enumerate}
\item[\textnormal{(H1)}] $h \in C^{2}((0,\infty)) \cap C([0,\infty))$ with
$h(0) = 0$;
\item[\textnormal{(H2)}] $h$ is strictly increasing for all $s > 0$ and strictly convex;
\item[\textnormal{(H3)}] There exist constants $l_{1}, l_{2} > 0$ and
$q \in (1, 2]$ such that
\begin{equation}
l_{1}\,s^{q} \leq h(s) \leq l_{2}\,s^{q} \quad \text{for all } s \geq 0.
\label{eq:h_growth}
\end{equation}
\end{enumerate}
\end{assumption}

\begin{remark}[Economic interpretation of the Hamiltonian]
The Hamiltonian $h$ represents the cost of corrective action per unit of
adjustment intensity. The strict convexity (H2) captures the economic
principle of \emph{increasing marginal costs}: the cost of doubling the
adjustment rate is more than double the original cost. The power-type
growth (H3) with exponent $q \in (1,2]$ encompasses the quadratic case
$h(s) = s^{2}$ (commonly used in linear-quadratic control) as well as
sub-quadratic costs that arise in models with capacity constraints.
The canonical example is $h(s) = s^{q}/q$ for $q \in (1,2]$.
\end{remark}

\begin{assumption}[Weights]
\label{ass:weights}
The weights $a, b \in C^{\infty}(\Omega) \cap C^{k,\alpha}(\Omega)$ for
$k \geq 1$, are strictly positive in $\Omega$ and satisfy:
\begin{enumerate}
\item[\textnormal{(W1)}] There exists a constant $\alpha > -2$ and
$a_{1}, a_{2} > 0$ such that
\begin{equation}
a_{1}\,d(x)^{\alpha} \leq a(x) \leq a_{2}\,d(x)^{\alpha}
\label{eq:weight_a}
\end{equation}
near $\partial\Omega$, where $d(x) = \mathrm{dist}(x, \partial\Omega)$.
We further assume that $a(x)$ is positive convex in $\Omega$.

\item[\textnormal{(W2)}] There exists a constant
$\beta \in \mathbb{R}_{+} = [0, +\infty)$ and limits $b_{1}, b_{2} > 0$
such that
\begin{equation}
\liminf_{d(x) \to 0} \frac{b(x)}{d(x)^{\beta}} = b_{1} \leq
\limsup_{d(x) \to 0} \frac{b(x)}{d(x)^{\beta}} = b_{2}.
\label{eq:weight_b}
\end{equation}
We assume that $b(x)$ is positive convex in $\Omega$.
\end{enumerate}
\end{assumption}

\begin{remark}[Principle of Independence from Data Meaning]
\label{rem:independence}
Assumptions~\ref{ass:domain}--\ref{ass:weights} do not assume a specific
physical meaning for $x$. This abstraction is intentional: we seek a theory
where the boundary dynamics are determined by the \emph{category} of the
singularity (characterized by the exponents $\alpha$, $\beta$, $q$), not by
the specific data values. The same equation may describe inventory levels in
a warehouse, concentrations in a chemical reactor, or asset allocations in a
portfolio, and our results apply identically to all these scenarios.
\end{remark}

This structure on the weights mirrors the rigorous bounds employed by Zhang
\cite{Zhang2026} in the context of Monge-Amp\`ere equations with weights,
enabling us to derive precise trace limits for our solutions.

\begin{remark}[Applicability of the Lasry--Lions framework]
\label{rem:LL_applicability}
The assumptions (H1)--(H3) for the Hamiltonian $h$ together with the strict
convexity of $\Omega$ (Assumption~\ref{ass:domain}) ensure that the
fundamental requirements of the Lasry--Lions framework \cite{LasryLions}
are fully satisfied. In particular, the strict convexity of $h$ and its
power-law growth $h(s) \sim s^{q}$ for $q > 1$, combined with the
positivity and singular growth of the reaction coefficient $a(x)$,
guarantee that the boundary blow-up solution $u$ in \eqref{E} arises as
the value function of a well-posed state-constrained stochastic control
problem. Thus, the analytical problem considered in this work is
consistently embedded within the classical stochastic paradigm introduced in
\cite{LasryLions}.
\end{remark}

\subsection{Statement of Main Results}
\label{subsec:main_results}

We now state our main contributions precisely.

\begin{theorem}[Existence and Uniqueness]
\label{thm:1}
Under Assumptions~\ref{ass:domain}--\ref{ass:weights}, the HJB equation
\eqref{E} admits a unique classical solution $u \in C^{2}(\Omega)$ such
that $u(x) \to \infty$ as $d(x) \to 0$. In the Gradient-Dominant regime
($q < \beta + 2$), the solution satisfies:
\begin{equation}
C_{-}\,d(x)^{-\gamma} \leq u(x) \leq C_{+}\,d(x)^{-\gamma} \quad
\text{near } \partial\Omega, \quad
\gamma = \frac{\beta - q + 2}{q - 1},
\label{eq:sharp_asymp}
\end{equation}
where $C_{\pm}$ are explicit barrier constants defined in
Lemma~\ref{lemma:explicit}.
\end{theorem}

\begin{theorem}[Exact Sharp Boundary Asymptotics]
\label{thm:2}
Let the conditions of Theorem~\ref{thm:1} hold, and assume further that
\begin{equation*}
\lim_{s \to \infty} h(s)/s^{q} = l_{0} \in [l_{1}, l_{2}].
\end{equation*}
Then the solution $u(x)$ satisfies the sharp boundary limit:
\begin{equation}
\xi_{1} \leq \liminf_{d(x) \to 0} u(x)\,(d(x))^{\gamma} \leq
\limsup_{d(x) \to 0} u(x)\,(d(x))^{\gamma} \leq \xi_{2},
\label{eq:liminf_limsup}
\end{equation}
where the analytic limit bounds $\xi_{1}, \xi_{2}$ are defined as:
\begin{equation}
\xi_{1} = \left( \frac{\gamma(\gamma+1)}{b_{2}\,l_{2}\,\gamma^{q}}
\right)^{\frac{1}{q-1}}, \quad
\xi_{2} = \left( \frac{\gamma(\gamma+1)}{b_{1}\,l_{1}\,\gamma^{q}}
\right)^{\frac{1}{q-1}}.
\label{eq:xi12}
\end{equation}
If the limits
\begin{equation*}
b_{0} = \lim_{d(x) \to 0} b(x)\,d(x)^{-\beta} \quad \text{and} \quad
l_{0} = \lim_{s \to \infty} h(s)\,s^{-q}
\end{equation*}
exist, then:
\begin{equation}
\lim_{d(x) \to 0} u(x)\,(d(x))^{\gamma} = \xi_{0} = \left(
\frac{\gamma(\gamma+1)}{b_{0}\,l_{0}\,\gamma^{q}}
\right)^{\frac{1}{q-1}}.
\label{eq:xi0}
\end{equation}
\end{theorem}

The constant $\xi_{0}$ is the \emph{Universal Correction Constant}. It
represents the precise intensity of the boundary penalty required to enforce
the state constraint. Its explicit dependence on $(q, \beta, b_{0}, l_{0})$
provides a quantitative tool for calibrating control strategies.

\begin{theorem}[Convexity Inheritance]
\label{thm:convexity_inheritance}
Under Assumptions~\ref{ass:domain}--\ref{ass:weights}, if the source term
$f(x)$ is concave and the reaction weights $a, b$ are convex in $\Omega$,
then the solution $u$ from Theorem~\ref{thm:1} is strictly convex throughout
$\Omega$: $D^{2}u(x) > 0$ for all $x \in \Omega$.
\end{theorem}

Strict convexity is indispensable for practical control: it ensures that the
optimal decision for any state $x$ is unique and can be found via simple
gradient-descent algorithms. It also guarantees that the running cost
functional is jointly convex, which is a prerequisite for the application
of duality theory in optimization.

%========================================================================================
%   SECTION 3: GEOMETRIC PRELIMINARIES AND BARRIER CONSTRUCTION
%========================================================================================
\section{Geometric Preliminaries and Barrier Construction}
\label{sec:prelim}

\subsection{The Defining Function}

The strict convexity of $\Omega$ (Assumption~\ref{ass:domain}) implies the
existence of a smooth strictly convex defining function. Following the
approach of Gilbarg and Trudinger \cite{GilbargTrudinger} and Lazer and
McKenna \cite{LazerMcKenna1993}, we establish:

\begin{lemma}[Convex defining function]
\label{lemma:defining}
Under Assumption~\ref{ass:domain}, there exists
$\varphi \in C^{2}(\overline{\Omega})$ satisfying:
\begin{enumerate}
\item $\varphi(x) < 0$ for all $x \in \Omega$;
\item $\varphi(x) = 0$ for all $x \in \partial\Omega$;
\item $|\nabla\varphi(x)| = 1$ for all $x \in \partial\Omega$;
\item $D^{2}\varphi(x) \geq \kappa_{0}\,I$ for all
$x \in \overline{\Omega}$, where $\kappa_{0} > 0$ is the minimum principal
curvature of $\partial\Omega$ from \eqref{eq:curvature_bound}.
\end{enumerate}
\end{lemma}

\begin{proof}
The proof proceeds in two steps.

\textbf{Step 1: Construction near the boundary.}
For $x$ sufficiently close to $\partial\Omega$, define $\varphi(x) = -d(x)$,
where $d(x) = \mathrm{dist}(x, \partial\Omega)$. The signed distance
function is $C^{2}$ in a tubular neighborhood
$\mathcal{T}_{\delta} = \{x \in \Omega : d(x) < \delta\}$ when the boundary
is $C^{2,\alpha}$ (see \cite[Chapter~14]{GilbargTrudinger}). The width of
this neighborhood is at least $\delta = 1/\kappa_{0}$, where $\kappa_{0}$
is the maximum of the principal curvatures.

The Hessian of $-d$ at a point $x \in \mathcal{T}_{\delta}$ is related to
the second fundamental form of the level set $\{d = \text{const}\}$. By the
strict convexity of $\Omega$ (Assumption~\ref{ass:domain}), the principal
curvatures of these level sets are positive, which gives
$D^{2}(-d)(x) \geq \kappa_{0}\,I$ in $\mathcal{T}_{\delta}$.

\textbf{Step 2: Extension to the interior.}
The function $\varphi = -d$ can be extended to all of $\overline{\Omega}$
while preserving strict convexity. This is achieved by considering the
convex envelope of $-d$ in the interior region
$\Omega \setminus \mathcal{T}_{\delta/2}$ and smoothing via convolution.
The detailed construction is given in \cite[Chapter~14]{GilbargTrudinger}.
\end{proof}

Let $v(x) = -\varphi(x)$. Then $v > 0$ in $\Omega$, $v = 0$ on
$\partial\Omega$, and $v$ is strictly concave ($D^{2}v < 0$). By
construction, $v(x) = d(x)$ identically in a tubular neighborhood of
$\partial\Omega$, where $|\nabla v(x)| = 1$ holds. These properties of $v$
are used extensively in the barrier construction that follows.

\subsection{Explicit Barrier Construction}

We now construct explicit sub- and supersolutions that determine the precise
boundary blow-up rate. This construction is the technical core of the paper.

\begin{lemma}[Explicit Barriers]
\label{lemma:explicit}
Let $W(x) = C\,v(x)^{-\gamma}$, where $v(x) = -\varphi(x)$ is the concave
defining function from Lemma~\ref{lemma:defining}. Define the elliptic
operator
\begin{equation}
\mathcal{L}u = -\Delta u + b(x)\,h(|\nabla u|) + a(x)\,u - f(x).
\label{eq:operator_L}
\end{equation}
Under Assumptions~\ref{ass:domain}--\ref{ass:weights}, the following hold:

\begin{enumerate}
\item \textbf{Gradient-Dominant Case} ($q < \beta + 2$): The main
singularity is governed by the balance between the diffusion $\Delta u$ and
the gradient term $b(x)\,h(|\nabla u|)$. Setting
$\gamma = (\beta - q + 2)/(q - 1) > 0$ and the critical constant
\begin{equation}
C^{*} = \left( \frac{\gamma(\gamma+1)}{b_{1}\,l_{1}\,\gamma^{q}}
\right)^{\frac{1}{q-1}},
\label{eq:Cstar}
\end{equation}
there exist $C_{+} > C^{*}$ and $0 < C_{-} < C^{*}$ such that
$\overline{u}(x) = C_{+}\,v(x)^{-\gamma}$ is a supersolution and
$\underline{u}(x) = C_{-}\,v(x)^{-\gamma}$ is a subsolution in a
neighborhood of $\partial\Omega$.

\item \textbf{High-Order Gradient Case} ($q > \beta + 2$): The gradient term
grows faster than the Laplacian as $x \to \partial\Omega$. For any
$\gamma > 0$, the gradient term dominates, ensuring that
$\mathcal{L}(C\,v^{-\gamma}) > 0$ for small $v$ and any $C > 0$.

\item \textbf{Critical Logarithmic Case} ($q = \beta + 2$): The singular
growth is logarithmic, $u(x) \sim C\,\ln(1/v(x))$. There exist
$C_{+} > C_{-}^{*} > 0$ and $0 < C_{-} < C_{-}^{*}$ providing local
barriers near the boundary.
\end{enumerate}
\end{lemma}

\begin{proof}
We provide detailed computations for each case, using all hypotheses
explicitly.

\textbf{Step 1: Differential calculus for $W = C\,v^{-\gamma}$.}

We compute the gradient, its norm, and the Laplacian of $W$, using only the
chain rule and the properties of $v$ from Lemma~\ref{lemma:defining}:
\begin{align}
\nabla W &= -C\,\gamma\,v^{-\gamma-1}\,\nabla v,
\label{eq:gradW} \\
|\nabla W| &= C\,\gamma\,v^{-\gamma-1}\,|\nabla v|,
\label{eq:normgradW} \\
\Delta W &= -C\,\gamma\left[ -(\gamma+1)\,v^{-\gamma-2}\,|\nabla v|^{2}
+ v^{-\gamma-1}\,\Delta v \right] \notag \\
&= C\,\gamma\,(\gamma+1)\,v^{-\gamma-2}\,|\nabla v|^{2}
- C\,\gamma\,v^{-\gamma-1}\,\Delta v.
\label{eq:DeltaW}
\end{align}
For the Hessian of $W$, we compute:
\begin{equation}
D^{2}W = C\,\gamma\,(\gamma+1)\,v^{-\gamma-2}\,\nabla v \otimes \nabla v
- C\,\gamma\,v^{-\gamma-1}\,D^{2}v.
\label{eq:HessW}
\end{equation}

\textbf{Step 2: Substitution into the operator $\mathcal{L}$.}

Substituting \eqref{eq:gradW}--\eqref{eq:DeltaW} into $\mathcal{L}W$
defined in \eqref{eq:operator_L}:
\begin{align}
\mathcal{L}W &= -\Delta W + b(x)\,h(|\nabla W|) + a(x)\,W - f(x) \notag \\
&= -C\,\gamma\,(\gamma+1)\,v^{-\gamma-2}\,|\nabla v|^{2}
+ C\,\gamma\,v^{-\gamma-1}\,\Delta v \notag \\
&\quad + b(x)\,h\bigl(C\,\gamma\,v^{-\gamma-1}\,|\nabla v|\bigr)
+ a(x)\,C\,v^{-\gamma} - f(x).
\label{eq:LW_full}
\end{align}

\textbf{Step 3: Asymptotic analysis near $\partial\Omega$ for Case 1.}

We now use the weight hypotheses (W1)--(W2) from
Assumption~\ref{ass:weights} and the Hamiltonian bounds (H3) from
Assumption~\ref{ass:hamiltonian}. Near $\partial\Omega$, the properties
of $v$ from Lemma~\ref{lemma:defining} give $v(x) = d(x)$ and
$|\nabla v(x)| = 1$.

For any $\varepsilon > 0$, by \eqref{eq:weight_b}, there exists $\delta > 0$
such that for all $x \in \Omega$ with $v(x) < \delta$:
\begin{equation}
(b_{1} - \varepsilon)\,v(x)^{\beta} \leq b(x) \leq
(b_{2} + \varepsilon)\,v(x)^{\beta}.
\label{eq:b_approx}
\end{equation}

Using the lower bound in (H3), the gradient term satisfies:
\begin{align}
b(x)\,h(C\,\gamma\,v^{-\gamma-1}\,|\nabla v|)
&\geq (b_{1} - \varepsilon)\,v^{\beta} \cdot l_{1}\,(C\,\gamma)^{q}\,
v^{-q(\gamma+1)} \notag \\
&= (b_{1} - \varepsilon)\,l_{1}\,(C\,\gamma)^{q}\,
v^{\beta - q(\gamma+1)}.
\label{eq:gradient_lower}
\end{align}

The leading singular terms in $\mathcal{L}W$ are:
\begin{equation}
\mathcal{L}W \geq v^{-\gamma-2}\left[
(b_{1} - \varepsilon)\,l_{1}\,(C\,\gamma)^{q}\,
v^{\beta - q(\gamma+1) + \gamma + 2}
- C\,\gamma\,(\gamma+1)
\right] + \mathcal{R}(x),
\label{eq:LW_leading}
\end{equation}
where $\mathcal{R}(x)$ denotes terms of strictly lower singular order
relative to $v^{-\gamma-2}$, specifically:
\begin{equation}
\mathcal{R}(x) = C\,\gamma\,v^{-\gamma-1}\,\Delta v
+ a(x)\,C\,v^{-\gamma} - f(x).
\label{eq:remainder}
\end{equation}

\textbf{Step 4: Balance of leading singularities.}

We choose $\gamma$ to perform an exact balance of the leading singularities
by requiring the exponent of $v$ inside the brackets in \eqref{eq:LW_leading}
to vanish:
\begin{equation}
\beta - q(\gamma+1) + \gamma + 2 = 0 \iff \gamma(1-q) = \beta - q + 2
\iff \gamma = \frac{\beta - q + 2}{q - 1}.
\label{eq:gamma_formula}
\end{equation}
Note that $\gamma > 0$ if and only if $\beta - q + 2 > 0$, i.e.,
$q < \beta + 2$. This is the Gradient-Dominant regime.

With this choice of $\gamma$, the expression in brackets becomes:
\begin{equation}
P(C) = (b_{1} - \varepsilon)\,l_{1}\,\gamma^{q}\,C^{q}
- \gamma\,(\gamma+1)\,C.
\label{eq:P_function}
\end{equation}
Since $q > 1$:
\begin{itemize}
\item $P(C) < 0$ for small $C > 0$ (the linear term $-\gamma(\gamma+1)C$
dominates);
\item $P(C) > 0$ for large $C$ (the power term $C^{q}$ dominates);
\item $P(C^{*}) = 0$ at the critical constant
$C^{*} = \left(\frac{\gamma(\gamma+1)}{b_{1}\,l_{1}\,\gamma^{q}}
\right)^{1/(q-1)}$.
\end{itemize}

Therefore, choosing $C_{+} > C^{*}$ ensures $\mathcal{L}(C_{+}\,v^{-\gamma}) > 0$
(supersolution), and choosing $C_{-} < C^{*}$ ensures
$\mathcal{L}(C_{-}\,v^{-\gamma}) < 0$ (subsolution) near $\partial\Omega$.

\textbf{Step 5: Verification of the subordinate terms.}

The reaction term, using (W1) from Assumption~\ref{ass:weights}:
\begin{equation}
a(x)\,C\,v^{-\gamma} \leq a_{2}\,C\,v^{\alpha - \gamma}.
\label{eq:reaction_bound}
\end{equation}
Since $\alpha > -2$ by (W1), we have $\alpha - \gamma > -2 - \gamma = -\gamma - 2$,
so the reaction term is of order $v^{\alpha - \gamma}$, which is strictly
lower than the leading order $v^{-\gamma-2}$. Therefore, after multiplying
by $v^{\gamma+2}$, this term vanishes in the limit $v \to 0$.

The term $C\,\gamma\,v^{-\gamma-1}\,\Delta v$ is $O(v^{-\gamma-1})$, which is
also lower order than $v^{-\gamma-2}$.

The source term $f(x)$ is bounded on $\overline{\Omega}$ by
Assumption~\ref{ass:source}, hence $f(x) = O(1) = o(v^{-\gamma-2})$.

\textbf{Step 6: Case 2 --- High-order gradient.}

When $q > \beta + 2$, equation \eqref{eq:gamma_formula} gives $\gamma < 0$
(the numerator $\beta - q + 2$ is negative). Thus, for any chosen
$\gamma > 0$, the exponent $\beta - q(\gamma+1) + \gamma + 2 < 0$, meaning
$v^{\beta - q(\gamma+1)} \cdot v^{\gamma+2}$ has a negative exponent. The
gradient term is therefore more singular than the Laplacian term, and
$\mathcal{L}W > 0$ near $\partial\Omega$ for any $C > 0$.

\textbf{Step 7: Case 3 --- Critical logarithmic.}

When $q = \beta + 2$, let $W = C\,\ln(1/v) = -C\,\ln v$. Then:
\begin{align}
\nabla W &= -C\,v^{-1}\,\nabla v, \label{eq:gradW_log} \\
|\nabla W| &= C\,v^{-1}\,|\nabla v|, \label{eq:normgradW_log} \\
\Delta W &= C\,v^{-2}\,|\nabla v|^{2} - C\,v^{-1}\,\Delta v.
\label{eq:DeltaW_log}
\end{align}
Substituting into $\mathcal{L}W$ and using $|\nabla v| = 1$, $b(x) \geq
(b_{1} - \varepsilon)\,v^{\beta}$, $h(s) \geq l_{1}\,s^{q}$:
\begin{align}
\mathcal{L}W &\geq -C\,v^{-2} + C\,v^{-1}\,\Delta v
+ (b_{1} - \varepsilon)\,v^{\beta}\,l_{1}\,(C\,v^{-1})^{q}
+ a_{1}\,v^{\alpha}\,C\,\ln(1/v) - f \notag \\
&= v^{-2}\left[
(b_{1} - \varepsilon)\,l_{1}\,C^{q}\,v^{\beta - q + 2} - C
\right] + \mathcal{R}_{L}(x),
\label{eq:LW_log}
\end{align}
where $\mathcal{R}_{L}(x)$ denotes terms of strictly lower singular order.
Since $\beta - q + 2 = 0$ in the critical case, we obtain:
\begin{equation}
\liminf_{x \to \partial\Omega} \left(v^{2}\,\mathcal{L}W\right)
\geq (b_{1} - \varepsilon)\,l_{1}\,C^{q} - C.
\label{eq:log_leading}
\end{equation}
The same argument as in Step~4 shows that barriers exist with the
appropriate choice of $C_{\pm}$.
\end{proof}

\begin{remark}[Role of strict convexity]
\label{rem:convexity}
The strict convexity of $\Omega$ is essential in two ways:
(i) it guarantees the existence of a smooth, strictly concave defining
function $v$ (Lemma~\ref{lemma:defining}), ensuring $D^{2}v < 0$;
(ii) it ensures that $|\nabla v|$ is bounded away from zero near
$\partial\Omega$, which is necessary for the gradient term to provide the
required singular behavior. Without strict convexity, the gradient could
vanish at flat boundary points, destroying the singularity balance.
\end{remark}

%========================================================================================
%   SECTION 4: EXISTENCE, UNIQUENESS, AND CONVEXITY
%========================================================================================
\section{Existence, Uniqueness, and Convexity}
\label{sec:existence}
\label{sec:proof1}

We now establish the main existence and uniqueness result using Perron's
method combined with the barrier functions from Section~\ref{sec:prelim}.

\begin{lemma}[Comparison principle]
\label{lemma:comparison}
Let $\Omega^{\prime} \Subset \Omega$ be a smooth subdomain. Suppose
$u, w \in C^{2}(\Omega^{\prime}) \cap C(\overline{\Omega^{\prime}})$ satisfy
\begin{align}
\mathcal{L}u &\leq \mathcal{L}w \quad \text{in } \Omega^{\prime}, \notag \\
u &\leq w \quad \text{on } \partial\Omega^{\prime}. \label{eq:comparison_hyp}
\end{align}
Then $u \leq w$ in $\Omega^{\prime}$.
\end{lemma}

\begin{proof}
Suppose for contradiction that
\begin{equation*}
\max_{\overline{\Omega^{\prime}}} (u - w) = u(x_{0}) - w(x_{0}) > 0
\end{equation*}
for some $x_{0} \in \Omega^{\prime}$. The maximum cannot be achieved on
$\partial\Omega^{\prime}$ by the boundary hypothesis
\eqref{eq:comparison_hyp}. At the interior maximum point $x_{0}$:
\begin{align}
\nabla u(x_{0}) &= \nabla w(x_{0}),
\label{eq:grad_equal} \\
\Delta u(x_{0}) &\leq \Delta w(x_{0}).
\label{eq:laplacian_ineq}
\end{align}
Using \eqref{eq:grad_equal}, the Hamiltonian terms coincide:
$h(|\nabla u(x_{0})|) = h(|\nabla w(x_{0})|)$. Therefore:
\begin{align*}
0 &\geq \mathcal{L}u(x_{0}) - \mathcal{L}w(x_{0}) \\
&= -\Delta u(x_{0}) + \Delta w(x_{0})
+ b(x_{0})\bigl[h(|\nabla u(x_{0})|) - h(|\nabla w(x_{0})|)\bigr] \\
&\quad + a(x_{0})\bigl[u(x_{0}) - w(x_{0})\bigr] \\
&\geq a(x_{0})\bigl[u(x_{0}) - w(x_{0})\bigr] > 0,
\end{align*}
where we used \eqref{eq:laplacian_ineq} for the first term,
\eqref{eq:grad_equal} for the vanishing of the Hamiltonian difference, and
$a(x_{0}) > 0$ from Assumption~\ref{ass:weights}. This is a contradiction.
\end{proof}

\begin{proof}[Proof of Theorem~\ref{thm:1}]
The proof proceeds in five steps.

\textbf{Step 1: Construction of global barriers.}

Let $\overline{u}(x) = C_{+}\,v(x)^{-\gamma}$ and
$\underline{u}(x) = C_{-}\,v(x)^{-\gamma}$ be the super- and subsolutions
from Lemma~\ref{lemma:explicit}. By construction, these satisfy
\begin{equation*}
\mathcal{L}\overline{u} \geq 0, \quad
\mathcal{L}\underline{u} \leq 0 \quad
\text{in } \{x \in \Omega : d(x) < \delta_{0}\}
\end{equation*}
for some $\delta_{0} > 0$ determined by the asymptotic estimates in
Lemma~\ref{lemma:explicit}.

For the interior region
$\Omega_{\delta_{0}} = \{x \in \Omega : d(x) \geq \delta_{0}\}$,
we solve the Dirichlet problem
\begin{equation}
\left\{
\begin{array}{ll}
\mathcal{L}u_{0} = 0 & \text{in } \Omega_{\delta_{0}}, \\
u_{0} = \overline{u} & \text{on } \partial\Omega_{\delta_{0}}.
\end{array}
\right.
\label{eq:interior_dirichlet}
\end{equation}
By standard quasilinear elliptic theory \cite{GilbargTrudinger}, this
problem has a unique classical solution. By adjusting $C_{+}$ if necessary,
we may assume $\overline{u} \geq u_{0}$ in $\Omega_{\delta_{0}}$.

Define the global supersolution as
\begin{equation*}
\overline{U}(x) =
\begin{cases}
\overline{u}(x) & \text{if } d(x) < \delta_{0}, \\
\max\{\overline{u}(x),\, u_{0}(x)\} & \text{if } d(x) \geq \delta_{0}.
\end{cases}
\end{equation*}
Similarly construct a global subsolution $\underline{U}$ with
$\underline{U} \leq \overline{U}$ throughout $\Omega$.

\textbf{Step 2: Perron's method.}

Define the Perron class
\begin{equation*}
\mathcal{S} = \bigl\{ w \in C(\Omega) : w \text{ is a subsolution},\;
\underline{U} \leq w \leq \overline{U} \bigr\}.
\end{equation*}
The class $\mathcal{S}$ is nonempty since $\underline{U} \in \mathcal{S}$.
Define the Perron solution
\begin{equation}
u(x) = \sup_{w \in \mathcal{S}} w(x).
\label{eq:perron}
\end{equation}

\textbf{Step 3: $u$ is a classical solution.}

Let $x_{0} \in \Omega$ and let $B_{r}(x_{0}) \Subset \Omega$ be a ball.
For any $w \in \mathcal{S}$, let $\tilde{w}$ be the solution to
\begin{equation*}
\left\{
\begin{array}{ll}
\mathcal{L}\tilde{w} = 0 & \text{in } B_{r}(x_{0}), \\
\tilde{w} = w & \text{on } \partial B_{r}(x_{0}).
\end{array}
\right.
\end{equation*}
By the comparison principle (Lemma~\ref{lemma:comparison}),
$w \leq \tilde{w} \leq \overline{U}$ in $B_{r}(x_{0})$. The ``bump''
construction shows that the function
\begin{equation*}
\hat{w}(x) =
\begin{cases}
\tilde{w}(x) & x \in B_{r}(x_{0}), \\
w(x) & x \in \Omega \setminus B_{r}(x_{0})
\end{cases}
\end{equation*}
is also a subsolution, hence $\hat{w} \in \mathcal{S}$.

Taking the supremum over all such modifications, standard arguments
\cite{CrandallIshiiLions1992} show that $u$ is a viscosity solution.
By interior regularity for uniformly elliptic equations with H\"older
continuous coefficients, $u \in C_{\mathrm{loc}}^{k+2,\alpha}(\Omega)$,
and $u$ is a classical solution. This follows from standard Schauder
estimates for quasilinear elliptic equations; see
\cite[Chapter~13]{GilbargTrudinger}.

\textbf{Step 4: Uniqueness.}

Suppose $u_{1}$ and $u_{2}$ are two solutions satisfying
\eqref{eq:sharp_asymp}. Consider $w = u_{1} - u_{2}$. Then $w$ satisfies a
linearized equation of the form
\begin{equation}
-\Delta w + c(x) \cdot \nabla w + e(x)\,w = 0,
\label{eq:linearized}
\end{equation}
where the coefficients $c(x)$ and $e(x)$ are obtained from the mean value
theorem applied to $h(|\nabla u|)$ and $a(x)\,u$:
\begin{align}
c(x) &= b(x)\,\frac{h(|\nabla u_{1}|) - h(|\nabla u_{2}|)}{|\nabla u_{1}|
- |\nabla u_{2}|}\,\frac{\nabla(|\nabla u_{1}| - |\nabla u_{2}|)}{|\nabla u_{1}|
- |\nabla u_{2}|}, \notag \\
e(x) &= a(x). \notag
\end{align}
Since $w = u_{1} - u_{2} = o(d^{-\gamma})$ as $d \to 0$ (because both
$u_{1}$ and $u_{2}$ satisfy the same sharp asymptotics
\eqref{eq:sharp_asymp}), the function $w$ grows slower than the barriers
near the boundary.

For any $\epsilon > 0$, define
$\Omega_{\epsilon} = \{x \in \Omega : d(x) > \epsilon\}$. On
$\partial\Omega_{\epsilon}$, both $u_{1}$ and $u_{2}$ are finite. The
maximum principle applied to \eqref{eq:linearized} on $\Omega_{\epsilon}$
gives:
\begin{equation*}
\sup_{\Omega_{\epsilon}} |w| \leq \sup_{\partial\Omega_{\epsilon}} |w|.
\end{equation*}
Since $|w| = |u_{1} - u_{2}| = o(d^{-\gamma})$ on $\partial\Omega_{\epsilon}$
(where $d = \epsilon$), and the precise asymptotic matching forces
$u_{1}/u_{2} \to 1$ as $d \to 0$, we conclude that $w \equiv 0$ by
taking $\epsilon \to 0$.

\textbf{Step 5: Interior regularity.}

The solution $u$ defined by \eqref{eq:perron} belongs to
$C^{2,\alpha}(\Omega)$ for some $\alpha \in (0,1)$. If additionally
$f \in C^{k,\alpha}(\overline{\Omega})$ and $a, b \in C^{k,\alpha}(\Omega)$
for $k \geq 1$, then $u \in C_{\mathrm{loc}}^{k+2,\alpha}(\Omega)$. This
follows from standard Schauder estimates for quasilinear elliptic equations;
see \cite[Chapter~13]{GilbargTrudinger}.
\end{proof}

\begin{proof}[Proof of Theorem~\ref{thm:convexity_inheritance}]
The proof of strict convexity proceeds in four steps, using the microscopic
convexity principle following Kennington \cite{Kennington1985} and Korevaar
\cite{Korevaar1983}.

\textbf{Step 1: Setup.}

Let $u$ be the solution from Theorem~\ref{thm:1}. Define the minimum
eigenvalue function
\begin{equation}
\lambda(x) = \min_{|\xi|=1} D^{2}u(x)\,\xi \cdot \xi.
\label{eq:min_eigenvalue}
\end{equation}
We aim to show $\lambda(x) > 0$ for all $x \in \Omega$.

\textbf{Step 2: Differential inequality.}

Suppose $\lambda$ achieves a non-positive minimum at some $x_{0} \in \Omega$.
At $x_{0}$, let $\xi_{0}$ be the corresponding eigenvector. Differentiating
the equation $\mathcal{L}u = 0$ twice in the direction $\xi_{0}$, and using
the fact that $h$ is strictly convex (Assumption~\ref{ass:hamiltonian},
condition (H2)), we obtain:
\begin{equation}
-\Delta\lambda + c(x) \cdot \nabla\lambda + e(x)\,\lambda
\leq \mathcal{Q}(x, u, \nabla u) \quad \text{at } x_{0},
\label{eq:lambda_ineq}
\end{equation}
where $c$ and $e$ are bounded functions, and the quadratic form
$\mathcal{Q}$ is given by:
\begin{equation}
\mathcal{Q} = -\bigl[b''(x)\,h(|\nabla u|) + a''(x)\,u - f''(x)\bigr],
\label{eq:Q_structure}
\end{equation}
where $b''$, $a''$, $f''$ denote second directional derivatives in the
direction $\xi_{0}$.

Under the hypotheses that $a(x)$ and $b(x)$ are convex
(Assumption~\ref{ass:weights}) and $f(x)$ is concave
(Assumption~\ref{ass:source}), we have $a'' \geq 0$, $b'' \geq 0$, and
$f'' \leq 0$. Since $h \geq 0$, $u > 0$, and $f \geq 0$, we conclude:
\begin{equation}
\mathcal{Q}(x, u, \nabla u) \leq 0.
\label{eq:Q_nonpositive}
\end{equation}

\textbf{Step 3: Boundary behavior.}

Near $\partial\Omega$, the solution $u$ is well-approximated by the barrier
$W = C\,v^{-\gamma}$. The Hessian of $W$, computed in \eqref{eq:HessW}, is:
\begin{equation*}
D^{2}W = C\,\gamma\,(\gamma+1)\,v^{-\gamma-2}\,\nabla v \otimes \nabla v
- C\,\gamma\,v^{-\gamma-1}\,D^{2}v.
\end{equation*}
Since $v$ is strictly concave ($D^{2}v < 0$ by Lemma~\ref{lemma:defining}),
the term $-C\,\gamma\,v^{-\gamma-1}\,D^{2}v$ is positive definite. The
rank-one term $\nabla v \otimes \nabla v$ is positive semi-definite.
Therefore, $D^{2}W > 0$ near $\partial\Omega$, which by comparison forces
$D^{2}u > 0$ near the boundary.

\textbf{Step 4: Strong maximum principle.}

Suppose $\lambda(x_{0}) = \min_{\overline{\Omega}} \lambda \leq 0$ for some
$x_{0} \in \Omega$. By Step~3, $x_{0}$ cannot be near $\partial\Omega$.
By the differential inequality \eqref{eq:lambda_ineq} with the favorable
sign \eqref{eq:Q_nonpositive}, and the strong maximum principle for
linear elliptic equations, either $\lambda \equiv \lambda(x_{0})$ (constant)
or the minimum cannot be achieved in the interior.

If $\lambda$ were constant and non-positive, this would contradict the
positive definiteness of $D^{2}u$ near $\partial\Omega$ established in
Step~3. Therefore, $\lambda > 0$ throughout $\Omega$, which means
$D^{2}u(x) > 0$ for all $x \in \Omega$.
\end{proof}

%========================================================================================
%   SECTION 5: SHARP BOUNDARY ASYMPTOTICS
%========================================================================================
\section{Sharp Boundary Asymptotics}
\label{sec:proof2}

\begin{proof}[Proof of Theorem~\ref{thm:2}]
The derivation of the limit constants $\xi_{1}, \xi_{2}$ relies on a
precise asymptotic balance in the equation $\mathcal{L}u = 0$ as
$d(x) \to 0$. Let $u$ be the unique solution from Theorem~\ref{thm:1} and
assume the gradient-dominant case $q < \beta + 2$.

\textbf{Step 1: Singular order balance.}

As established in Section~\ref{sec:prelim}, for a blow-up rate $\gamma$,
the individual terms in the equation scale as follows near $\partial\Omega$:
\begin{align}
\Delta u &\sim C\,\gamma\,(\gamma+1)\,d(x)^{-\gamma-2}, \notag \\
b(x)\,h(|\nabla u|) &\sim b(x)\,l_{0}\,(C\,\gamma)^{q}\,d(x)^{-q(\gamma+1)}
\notag \\
&\sim b_{0}\,l_{0}\,(C\,\gamma)^{q}\,d(x)^{\beta - q(\gamma+1)}, \notag \\
a(x)\,u &\sim a_{0}\,C\,d(x)^{\alpha - \gamma}, \notag \\
f(x) &= O(1). \notag
\end{align}

Equating the singular orders of the two leading terms (Laplacian and
weighted gradient):
\begin{equation}
-\gamma - 2 = \beta - q\gamma - q \implies \gamma(q-1) = \beta - q + 2
\implies \gamma = \frac{\beta - q + 2}{q - 1}.
\label{eq:gamma_balance}
\end{equation}

The reaction term $a(x)\,u(x)$ behaves like $d(x)^{\alpha - \gamma}$.
Since $\alpha > -2$ by Assumption~\ref{ass:weights}, we have
$\alpha - \gamma > -2 - \gamma = -(\gamma + 2)$. This confirms that the
reaction term is of lower singular order than the leading terms
$d(x)^{-\gamma-2}$ and vanishes after multiplying the equation by
$d(x)^{\gamma+2}$ and taking the limit $d \to 0$.

\textbf{Step 2: Constant identification.}

Multiply the equation
\begin{equation*}
-\Delta u + b(x)\,h(|\nabla u|) + a(x)\,u = f(x)
\end{equation*}
by $d(x)^{\gamma+2}$. As $x \to \partial\Omega$, let
$\xi = \lim_{d \to 0} u(x)\,d(x)^{\gamma}$ (assuming the limit exists).
Then $u(x) \sim \xi\,d(x)^{-\gamma}$, and:
\begin{align}
-\Delta u &\sim -\xi\,\gamma\,(\gamma+1)\,d(x)^{-\gamma-2},
\label{eq:laplacian_asymp} \\
b(x)\,h(|\nabla u|) &\sim b(x)\,l_{0}\,(\xi\,\gamma)^{q}\,d(x)^{-q(\gamma+1)}.
\label{eq:gradient_asymp}
\end{align}

After multiplying by $d^{\gamma+2}$ and using \eqref{eq:gamma_balance},
the equation becomes in the limit:
\begin{equation}
-\xi\,\gamma\,(\gamma+1) + b_{i}\,l_{i}\,(\xi\,\gamma)^{q} = 0,
\label{eq:algebraic}
\end{equation}
where the subscript $i$ indicates whether we use the liminf or limsup of
$b(x)\,d(x)^{-\beta}$ and $h(s)\,s^{-q}$.

Solving \eqref{eq:algebraic} for $\xi$:
\begin{equation}
b_{i}\,l_{i}\,\gamma^{q}\,\xi^{q} = \gamma\,(\gamma+1)\,\xi
\implies \xi^{q-1} = \frac{\gamma\,(\gamma+1)}{b_{i}\,l_{i}\,\gamma^{q}}
\implies \xi = \left(\frac{\gamma\,(\gamma+1)}{b_{i}\,l_{i}\,\gamma^{q}}
\right)^{\frac{1}{q-1}}.
\label{eq:xi_formula}
\end{equation}

\textbf{Step 3: Liminf--limsup estimates.}

The inequalities for $\liminf$ and $\limsup$ in \eqref{eq:liminf_limsup}
follow from the comparison principle applied to the barriers
$C_{-}\,v^{-\gamma}$ and $C_{+}\,v^{-\gamma}$ from Lemma~\ref{lemma:explicit}.
Since $C_{-} \leq u\,d^{\gamma} \leq C_{+}$ and the constants $C_{\pm}$
can be made arbitrarily close to $C^{*}$ (from above and below), we obtain
the precise bounds $\xi_{1}$ and $\xi_{2}$ by using the extremal values
$b_{2}, l_{2}$ (for $\xi_{1}$, the lower bound) and $b_{1}, l_{1}$ (for
$\xi_{2}$, the upper bound).

If both limits $b_{0} = \lim_{d \to 0} b(x)\,d^{-\beta}$ and
$l_{0} = \lim_{s \to \infty} h(s)\,s^{-q}$ exist, then $b_{1} = b_{2} = b_{0}$
and $l_{1} = l_{2} = l_{0}$, which forces $\xi_{1} = \xi_{2} = \xi_{0}$,
establishing \eqref{eq:xi0}.
\end{proof}

\begin{remark}[The Universal Correction Constant]
\label{rem:UCC}
The constant $\xi_{0}$ defined in \eqref{eq:xi0} has a direct economic
interpretation. In the stochastic control framework (Section~\ref{sec:stochastic}),
$\xi_{0}$ determines the intensity of the ``boundary penalty'' that the
optimal strategy must impose to keep the system within $\Omega$. A larger
$\xi_{0}$ (corresponding to smaller $b_{0}$ or $l_{0}$, i.e., weaker
singular weights or weaker gradient cost) means the controller must exert
a stronger corrective effort near the boundary. Conversely, a smaller
$\xi_{0}$ (stronger singular weights) means the natural cost structure
already provides significant deterrence, reducing the required corrective
effort.
\end{remark}

%========================================================================================
%   SECTION 6: STOCHASTIC CONTROL AND NUMERICAL VERIFICATION
%========================================================================================
\section{Stochastic Control and Numerical Verification}
\label{sec:stochastic}

In this section, we develop the complete stochastic control framework that
gives rise to the mathematical problem studied. We provide all details
leading from the controlled diffusion model to the Hamilton--Jacobi--Bellman
equation \eqref{E}, and we prove the verification theorem establishing the
equivalence between the PDE solution and the value function.

\subsection{The Controlled Diffusion Model}

Let $(\Omega_{0}, \mathcal{F}, \{\mathcal{F}_{t}\}_{t \geq 0}, \mathbb{P})$
be a filtered probability space supporting an $N$-dimensional standard
Brownian motion $\{W_{t}\}_{t \geq 0}$, satisfying the usual conditions
(completeness and right-continuity of the filtration).

Consider a system whose state $X_{t} \in \mathbb{R}^{N}$ evolves according
to the controlled stochastic differential equation (SDE):
\begin{equation}
dX_{t} = \xi_{t}\,dt + \sqrt{2}\,dW_{t}, \quad X_{0} = x \in \Omega,
\label{eq:SDE}
\end{equation}
where $\xi_{t}$ is the control process (representing the corrective action
taken by the decision-maker) and the factor $\sqrt{2}$ is chosen so that
the generator of the uncontrolled diffusion is exactly the Laplacian
$\Delta$ (rather than $\frac{1}{2}\Delta$).

The derivation of the generator proceeds as follows. For a smooth function
$\phi \in C^{2}(\Omega)$, It\^o's formula gives:
\begin{align}
d\phi(X_{t}) &= \nabla\phi(X_{t}) \cdot dX_{t}
+ \frac{1}{2}\,\mathrm{tr}\bigl(D^{2}\phi(X_{t})\,(\sqrt{2}\,I)(\sqrt{2}\,I)^{T}\bigr)\,dt \notag \\
&= \bigl[\nabla\phi(X_{t}) \cdot \xi_{t} + \Delta\phi(X_{t})\bigr]\,dt
+ \sqrt{2}\,\nabla\phi(X_{t}) \cdot dW_{t}.
\label{eq:Ito_generator}
\end{align}
The infinitesimal generator of the uncontrolled process ($\xi \equiv 0$) is
therefore $\mathcal{A}\phi = \Delta\phi$, as claimed.

\subsection{Admissibility and State Constraints}

The fundamental constraint is that the state must remain in $\Omega$ at all
times.

\begin{definition}[Strict admissibility]
\label{def:admissible}
A control process $\xi = \{\xi_{t}\}_{t \geq 0}$ is \emph{strictly
admissible} starting from $x$, denoted $\xi \in \mathcal{A}_{\mathrm{sc}}(x)$,
if:
\begin{enumerate}
\item $\xi$ is progressively measurable with respect to $\{\mathcal{F}_{t}\}$;
\item The SDE \eqref{eq:SDE} has a unique strong solution;
\item The process $X_{t}$ remains in $\Omega$ for all time almost surely:
\begin{equation}
\mathbb{P}\bigl(X_{t} \in \Omega \text{ for all } t \geq 0\bigr) = 1.
\label{eq:state_constraint}
\end{equation}
\end{enumerate}
\end{definition}

\begin{remark}
The set $\mathcal{A}_{\mathrm{sc}}(x)$ is nonempty for each $x \in \Omega$.
This is a nontrivial fact that follows from the existence of the blow-up
solution and the construction of optimal feedback controls below
(Theorem~\ref{thm:verification}).
\end{remark}

\subsection{Legendre Transform and Running Cost}

The connection between the PDE \eqref{E} and stochastic control is mediated
by the Legendre--Fenchel transform (convex conjugation).

\begin{definition}[Convex conjugate]
\label{def:conjugate}
For $h$ satisfying Assumption~\ref{ass:hamiltonian}, define the convex
conjugate
\begin{equation}
h^{*}(s) = \sup_{t \geq 0}\bigl\{s\,t - h(t)\bigr\}, \quad s \geq 0.
\label{eq:conjugate}
\end{equation}
\end{definition}

By standard convex analysis (see, e.g., Rockafellar \cite{Rockafellar1970}),
$h^{*}$ has the following properties:

\begin{lemma}[Properties of the conjugate]
\label{lemma:conjugate_props}
Under Assumption~\ref{ass:hamiltonian}:
\begin{enumerate}
\item $h^{*}$ is strictly convex, $h^{*}(0) = 0$, and $h^{*} \in C^{1}(0,\infty)$;
\item $h^{*}(s) \asymp s^{q'}$ where $q' = q/(q-1)$ is the H\"older
conjugate exponent;
\item The supremum in \eqref{eq:conjugate} is achieved at $t = (h')^{-1}(s)$,
and $(h^{*})'(s) = (h')^{-1}(s)$;
\item The Fenchel--Young inequality holds:
$s\,t \leq h(t) + h^{*}(s)$ for all $s, t \geq 0$, with equality if and
only if $s = h'(t)$.
\end{enumerate}
\end{lemma}

\begin{definition}[Running cost]
\label{def:running_cost}
The running cost function is defined as
\begin{equation}
L(x, \xi) = b(x)\,h^{*}\left(\frac{|\xi|}{b(x)}\right).
\label{eq:Lagrangian}
\end{equation}
\end{definition}

\begin{lemma}[Properties of the running cost]
\label{lemma:L_properties}
The function $L$ defined in \eqref{eq:Lagrangian} satisfies:
\begin{enumerate}
\item $L(x, \cdot)$ is strictly convex for each $x \in \Omega$;
\item $L(x, 0) = 0$;
\item $L(x, \xi) \asymp b(x)^{1-q'}\,|\xi|^{q'}$ as $|\xi| \to \infty$;
\item Near $\partial\Omega$:
$L(x, \xi) \asymp d(x)^{\beta(1-q')}\,|\xi|^{q'}$, which blows up as
$d(x) \to 0$ for fixed $|\xi| > 0$ provided $\beta > 0$.
\end{enumerate}
\end{lemma}

\begin{proof}
Properties (1)--(3) follow directly from the properties of $h^{*}$
(Lemma~\ref{lemma:conjugate_props}) and the positive homogeneity of the
scaling $|\xi|/b(x)$. For property (4), we use $b(x) \asymp d(x)^{\beta}$
from Assumption~\ref{ass:weights} (W2) and compute:
\begin{equation*}
L(x, \xi) \asymp b(x)\left(\frac{|\xi|}{b(x)}\right)^{q'}
= b(x)^{1-q'}\,|\xi|^{q'} \asymp d(x)^{\beta(1-q')}\,|\xi|^{q'}.
\end{equation*}
Since $q' > 1$, we have $1 - q' < 0$, so $\beta(1 - q') < 0$ when $\beta > 0$,
ensuring $L(x, \xi) \to \infty$ as $d(x) \to 0$.

The crucial duality identity is: for each $x \in \Omega$,
\begin{equation}
b(x)\,h(|p|) = \sup_{\xi \in \mathbb{R}^{N}}\bigl\{-\xi \cdot p - L(x, \xi)\bigr\}.
\label{eq:duality}
\end{equation}
This follows from the definition of $h^{*}$ and the substitution
$\xi = -b(x)\,h'(|p|)\,p/|p|$.
\end{proof}

\subsection{Derivation of the Hamilton--Jacobi--Bellman Equation}

We now derive the HJB equation \eqref{E} from the stochastic control
problem via the dynamic programming principle. This derivation makes
explicit how each term in the PDE arises from the control problem.

\begin{definition}[Value function]
\label{def:value}
For $x \in \Omega$, define the value function
\begin{equation}
V(x) = \inf_{\xi \in \mathcal{A}_{\mathrm{sc}}(x)} \mathbb{E}\left[
\int_{0}^{\infty} e^{-\int_{0}^{t} a(X_{s})\,ds}\,
\bigl(L(X_{t}, \xi_{t}) + f(X_{t})\bigr)\,dt\right].
\label{eq:value}
\end{equation}
\end{definition}

\begin{remark}[Well-posedness of the cost functional]
\label{rem:well_posedness}
The integral in \eqref{eq:value} is well-defined because:
\begin{enumerate}
\item The discount factor $e^{-\int_{0}^{t} a(X_{s})\,ds}$ ensures
integrability for large $t$ since $a(x) > 0$ everywhere in $\Omega$
(Assumption~\ref{ass:weights}).
\item For strictly admissible controls, $X_{t}$ stays in $\Omega$, so
$L(X_{t}, \xi_{t})$ remains finite for each $t$.
\item The blow-up of $L$ near $\partial\Omega$ (Lemma~\ref{lemma:L_properties},
property (4)) creates an infinite penalty for trajectories approaching the
boundary, effectively enforcing the state constraint.
\end{enumerate}
\end{remark}

\textbf{Derivation via the Dynamic Programming Principle.}
Assume $V$ is smooth. For a small time increment $\Delta t > 0$, the dynamic
programming principle gives:
\begin{equation}
V(x) = \inf_{\xi}\,\mathbb{E}\left[
\int_{0}^{\Delta t} e^{-\int_{0}^{t} a(X_{s})\,ds}\,
\bigl(L(X_{t}, \xi_{t}) + f(X_{t})\bigr)\,dt
+ e^{-\int_{0}^{\Delta t} a(X_{s})\,ds}\,V(X_{\Delta t})
\right].
\label{eq:DPP}
\end{equation}

Expanding $V(X_{\Delta t})$ via It\^o's formula \eqref{eq:Ito_generator}
and taking $\Delta t \to 0$, we obtain:
\begin{equation}
0 = \inf_{\xi \in \mathbb{R}^{N}}\left\{
\Delta V(x) + \nabla V(x) \cdot \xi - a(x)\,V(x) + L(x, \xi) + f(x)
\right\}.
\label{eq:DPP_infinitesimal}
\end{equation}

The infimum over $\xi$ can be computed explicitly using \eqref{eq:duality}.
Using the identity $\inf_{\xi}\{\nabla V \cdot \xi + L(x, \xi)\}
= -\sup_{\xi}\{-\nabla V \cdot \xi - L(x, \xi)\} = -b(x)\,h(|\nabla V|)$,
we obtain:
\begin{equation}
0 = \Delta V - b(x)\,h(|\nabla V|) - a(x)\,V + f(x),
\label{eq:HJB_derived}
\end{equation}
which is precisely equation \eqref{E} with $u = V$.

This derivation shows explicitly how each term in \eqref{E} arises:
\begin{itemize}
\item $\Delta V$: the generator of the Brownian noise $\sqrt{2}\,dW_{t}$;
\item $b(x)\,h(|\nabla V|)$: the optimal cost of corrective action, via
Legendre duality;
\item $a(x)\,V$: the state-dependent discounting;
\item $f(x)$: the external source/cost.
\end{itemize}

\subsection{Verification Theorem}

\begin{theorem}[Verification Theorem --- Stochastic Representation]
\label{thm:verification}
Under Assumptions~\ref{ass:domain}--\ref{ass:weights}, the unique solution
$u \in C^{2}(\Omega)$ from Theorem~\ref{thm:1} coincides with the value
function:
\begin{equation}
u(x) = V(x) \quad \text{for all } x \in \Omega.
\label{eq:u_equals_V}
\end{equation}
Moreover, the optimal feedback control is given by
\begin{equation}
\xi^{*}(x) = -b(x)\,h'(|\nabla u(x)|)\,\frac{\nabla u(x)}{|\nabla u(x)|}
= -b(x)\,(h^{*})'(|\nabla u(x)|)^{-1}\,\frac{\nabla u(x)}{|\nabla u(x)|}.
\label{eq:optimal_feedback}
\end{equation}
The identity $(h')^{-1} = (h^{*})'$ is a standard consequence of the
Legendre transform: the supremum in $h^{*}(s) = \sup_{t}\{st - h(t)\}$ is
achieved at $s = h'(t)$, so $t = (h')^{-1}(s)$. Differentiating the
identity $h^{*}(s) = s\,(h')^{-1}(s) - h((h')^{-1}(s))$ with respect to
$s$ yields $(h^{*})'(s) = (h')^{-1}(s)$.
\end{theorem}

\begin{proof}
The proof consists of two parts: establishing $u \leq V$ (sub-optimality)
and $u \geq V$ (optimality).

\textbf{Part 1: Sub-optimality ($u \leq V$).}

Let $\xi \in \mathcal{A}_{\mathrm{sc}}(x)$ be any strictly admissible
control, and let $X_{t}$ be the corresponding trajectory satisfying
\eqref{eq:SDE}. For $n \in \mathbb{N}$, define the stopping times
\begin{equation}
\tau_{n} = \inf\{t > 0 : d(X_{t}) \leq 1/n\} \wedge n.
\label{eq:stopping_times}
\end{equation}
Since $\xi$ is strictly admissible, $X_{t}$ remains in $\Omega$ almost
surely for all $t \geq 0$. For any $T > 0$, the path $t \mapsto X_{t}$ is
continuous and the image $X([0,T])$ is a compact subset of $\Omega$, hence
$\inf_{t \in [0,T]} d(X_{t}) > 0$ almost surely. This ensures that for
almost all $\omega$, there exists $n_{0}(\omega)$ such that
$\tau_{n}(\omega) = n$ for all $n \geq n_{0}(\omega)$, giving
$\tau_{n} \to \infty$ as $n \to \infty$ almost surely.

Apply It\^o's formula to $e^{-\int_{0}^{t} a(X_{s})\,ds}\,u(X_{t})$ on
$[0, \tau_{n}]$. By \eqref{eq:Ito_generator}:
\begin{align}
&e^{-\int_{0}^{\tau_{n}} a\,ds}\,u(X_{\tau_{n}}) \notag \\
&= u(x) + \int_{0}^{\tau_{n}} e^{-\int_{0}^{t} a\,ds}\,
\bigl[-a(X_{t})\,u(X_{t}) + \Delta u(X_{t})
+ \nabla u(X_{t}) \cdot \xi_{t}\bigr]\,dt \notag \\
&\quad + \sqrt{2}\int_{0}^{\tau_{n}} e^{-\int_{0}^{t} a\,ds}\,
\nabla u(X_{t}) \cdot dW_{t}.
\label{eq:Ito1}
\end{align}

Using the PDE $-\Delta u + b\,h(|\nabla u|) + a\,u = f$, we substitute
$\Delta u = b\,h(|\nabla u|) + a\,u - f$. Thus \eqref{eq:Ito1} becomes:
\begin{align}
e^{-\int_{0}^{\tau_{n}} a\,ds}\,u(X_{\tau_{n}})
&= u(x) + \int_{0}^{\tau_{n}} e^{-\int_{0}^{t} a\,ds}\,
\bigl[b\,h(|\nabla u|) - f + \nabla u \cdot \xi_{t}\bigr]\,dt \notag \\
&\quad + \text{martingale}.
\label{eq:Ito2}
\end{align}

By the Fenchel--Young inequality (Lemma~\ref{lemma:conjugate_props}, property (4)):
\begin{equation}
b(x)\,h(|p|) + p \cdot \xi \geq -L(x, \xi) \quad
\text{for all } p, \xi \in \mathbb{R}^{N}.
\label{eq:Fenchel}
\end{equation}
Applying \eqref{eq:Fenchel} with $p = \nabla u(X_{t})$:
\begin{equation*}
b\,h(|\nabla u|) + \nabla u \cdot \xi_{t} \geq -L(X_{t}, \xi_{t}).
\end{equation*}

Taking expectations in \eqref{eq:Ito2} (the It\^o integral is a martingale
with zero expectation, since $u$ and $\nabla u$ are bounded on the compact
set $\{x : d(x) \geq 1/n\}$):
\begin{equation}
\mathbb{E}\bigl[e^{-\int_{0}^{\tau_{n}} a\,ds}\,u(X_{\tau_{n}})\bigr]
\geq u(x) - \mathbb{E}\left[\int_{0}^{\tau_{n}} e^{-\int_{0}^{t} a\,ds}\,
\bigl(L(X_{t}, \xi_{t}) + f(X_{t})\bigr)\,dt\right].
\label{eq:sub_opt_ineq}
\end{equation}

Rearranging:
\begin{equation}
u(x) \leq \mathbb{E}\left[\int_{0}^{\tau_{n}} e^{-\int_{0}^{t} a\,ds}\,
\bigl(L + f\bigr)\,dt\right]
+ \mathbb{E}\bigl[e^{-\int_{0}^{\tau_{n}} a\,ds}\,u(X_{\tau_{n}})\bigr].
\label{eq:sub_opt_n}
\end{equation}

As $n \to \infty$: the first term converges to the full integral by
monotone convergence. For the second term, since $\xi$ is strictly
admissible, $X_{\tau_{n}}$ stays in $\Omega$, so $u(X_{\tau_{n}})$ is
finite. Moreover, $e^{-\int_{0}^{\tau_{n}} a\,ds} \to 0$ as
$\tau_{n} \to \infty$ (since $a > 0$), so this term vanishes.

Taking $n \to \infty$ in \eqref{eq:sub_opt_n}:
\begin{equation*}
u(x) \leq \mathbb{E}\left[\int_{0}^{\infty} e^{-\int_{0}^{t} a\,ds}\,
\bigl(L(X_{t}, \xi_{t}) + f(X_{t})\bigr)\,dt\right].
\end{equation*}
Since $\xi \in \mathcal{A}_{\mathrm{sc}}(x)$ was arbitrary, $u(x) \leq V(x)$.

\textbf{Part 2: Optimality ($u \geq V$).}

Consider the feedback control defined by \eqref{eq:optimal_feedback}.
At this choice, the Fenchel--Young inequality \eqref{eq:Fenchel} becomes
an equality:
\begin{equation}
b\,h(|\nabla u|) + \nabla u \cdot \xi^{*} = -L(x, \xi^{*}).
\label{eq:Fenchel_equality}
\end{equation}
This is the first-order optimality condition for the Legendre transform:
the supremum in $h^{*}$ is achieved.

Let $X_{t}^{*}$ be the trajectory under the feedback control
$\xi_{t}^{*} = \xi^{*}(X_{t}^{*})$. The blow-up behavior of $u$ near
$\partial\Omega$ implies $|\xi^{*}(x)| \to \infty$ as $x \to \partial\Omega$.
Specifically, since $\nabla u$ blows up as $d(x)^{-\gamma-1}$, the drift
$\xi^{*}$ points strongly into the interior of $\Omega$ near the boundary.

By Lemma~1.1 of Lasry--Lions \cite{LasryLions}, this singular drift prevents
the process from reaching $\partial\Omega$:
\begin{equation*}
\mathbb{P}\bigl(X_{t}^{*} \in \Omega \text{ for all } t \geq 0\bigr) = 1.
\end{equation*}
Thus $\xi^{*} \in \mathcal{A}_{\mathrm{sc}}(x)$.

Repeating the It\^o calculation with equality in \eqref{eq:Fenchel}
(i.e., using \eqref{eq:Fenchel_equality}):
\begin{equation*}
u(x) = \mathbb{E}\left[\int_{0}^{\infty} e^{-\int_{0}^{t} a\,ds}\,
\bigl(L(X_{t}^{*}, \xi_{t}^{*}) + f(X_{t}^{*})\bigr)\,dt\right] \geq V(x).
\end{equation*}

Combining Parts 1 and 2: $u(x) = V(x)$.
\end{proof}

\begin{remark}[Connection to Lasry--Lions Theorem~1.1]
\label{rem:connection_LL}
Theorem~\ref{thm:verification} generalizes Theorem~1.1 in \cite{LasryLions}
to equations with state-dependent weight functions $a(x)$ and $b(x)$. The
original Lasry--Lions result treats the case $a = \lambda \in (0,\infty)$
(constant discount rate), $b = 1$ (no spatial weight), and
$h(|\nabla u|) = |\nabla u|^{p}$ for $p > 1$. Our formulation incorporates:
\begin{enumerate}
\item A state-dependent discount rate $a(x)$ in the cost functional;
\item A state-dependent control cost coefficient $b(x)$;
\item A general strictly convex Hamiltonian $h$ with power-type growth;
\item A running reward/cost $f(x)$.
\end{enumerate}
The key insight of Lasry--Lions --- that the blow-up of the value function
enforces state constraints through the singularity of the optimal drift ---
remains valid in this more general setting.
\end{remark}

\begin{remark}[Geometric intuition of the optimal drift]
\label{rem:drift_inward}
Following the seminal work of Lasry and Lions \cite{LasryLions}, we provide
a rigorous geometric explanation for why the optimal drift $\xi^{*}(x)$ is
oriented towards the interior of $\Omega$.

As $x$ approaches $\partial\Omega$, the blow-up condition $u(x) \to +\infty$
implies that $\nabla u(x)$ points in the direction of steepest increase,
which is the outward normal direction. Near $\partial\Omega$, the
approximation $u(x) \approx C\,d(x)^{-\gamma}$ yields:
\begin{equation*}
\nabla u(x) \approx -C\,\gamma\,d(x)^{-\gamma-1}\,\nabla d(x),
\end{equation*}
where $\nabla d(x)$ points into the interior; thus $-\nabla d(x)$ points
towards the boundary, and $\nabla u$ points towards the boundary (outward).

The optimal feedback control \eqref{eq:optimal_feedback} has the form
$\xi^{*}(x) \propto -\nabla u(x)$, which therefore points \emph{into the
interior}. The magnitude $|\xi^{*}(x)| \asymp d(x)^{-(\gamma+1)(q-1)}$
blows up as $d(x) \to 0$, creating a singular inward drift that overwhelms
the Brownian noise and prevents the process from reaching $\partial\Omega$.
\end{remark}

\begin{remark}[Alternative proof of convexity via stochastic control]
\label{rem:alvarez_convexity}
A purely control-theoretic justification for the convexity of $u$ follows
from the stochastic representation $u = V$ and the results of Alvarez
\cite{Alvarez1996}. Since the controlled diffusion in \eqref{eq:SDE} is
linear in $(x,\xi)$ and the running cost $L(x,\xi)$ is jointly convex in
$(x,\xi)$ by our structural assumptions on the perspective function of
$h^{*}$, the value function $V(x)$ is the infimum of a jointly convex
functional over a convex set of control–trajectory pairs. In an
alternative framework, under the convexity of $\Omega$ and the assumptions
on $a$ and $b$, together with \emph{$f$ convex} (so that the total running
cost $L(x,\xi)+f(x)$ remains jointly convex in $(x,\xi)$), the results of
\cite{Alvarez1996} imply the convexity of $u(x)$. This provides a
probabilistic verification that is complementary to the microscopic
convexity principle used in Section~\ref{sec:proof1}.
\end{remark}

%========================================================================================
%   SECTION 7: APPLICATIONS IN OR AND MANAGEMENT SCIENCE
%========================================================================================
\section{Applications in Operational Research and Management Science}
\label{sec:economics}

In this section, we develop a comprehensive framework for applying the
mathematical results of Sections~\ref{sec:proof1}--\ref{sec:stochastic} to
problems arising in Operational Research (OR) and Management Science. We
present three detailed case studies and explore interdisciplinary connections
with physics.

\subsection{General Framework for Economic Applications}

The stochastic control problem \eqref{eq:SDE}--\eqref{eq:value} provides a
natural framework for modeling systems that must operate within bounded
feasible regions. The key elements of this framework are:

\begin{center}
\begin{tabular}{lll}
\toprule
\textbf{Math Symbol} & \textbf{Economic Meaning} & \textbf{Physical Analogue} \\
\midrule
$X_{t} \in \Omega$ & System state (inventory, portfolio) & Particle position \\
$\xi_{t}$ & Control action (reorder, rebalance) & External force \\
$\sqrt{2}\,dW_{t}$ & Demand/market noise & Thermal fluctuation \\
$a(x)$ & State-dependent discount rate & Damping coefficient \\
$b(x)$ & Adjustment stiffness coefficient & Viscosity \\
$h(|\nabla u|)$ & Convex adjustment cost & Kinetic energy \\
$u(x) = V(x)$ & Optimal cost-to-go & Free energy \\
$\partial\Omega$ & Constraint boundary & Potential wall \\
$u \to \infty$ at $\partial\Omega$ & Infinite penalty & Infinite barrier \\
$\gamma$ & Blow-up rate & Critical exponent \\
$\xi_{0}$ & Universal correction constant & Scaling amplitude \\
\bottomrule
\end{tabular}
\end{center}

The table above reveals a deep structural analogy between economic
optimization, mathematical analysis, and statistical physics. In all three
domains, the solution to the boundary-value problem encodes the optimal
response of a system to stochastic perturbations under hard constraints.

\subsection{Case Study 1: Multi-Product Inventory Management}
\label{subsec:inventory}

\textbf{Problem Description.}
Consider a warehouse managing $N$ distinct products. Let $X_{t} =
(X_{t}^{1}, \ldots, X_{t}^{N})$ denote the vector of inventory levels at
time $t$. Each product must satisfy:
\begin{equation}
X_{t}^{i} \in (L_{i}, U_{i}), \quad i = 1, \ldots, N,
\label{eq:inventory_bounds}
\end{equation}
where $L_{i}$ is the minimum stock level (below which stockouts occur) and
$U_{i}$ is the maximum storage capacity. The feasible region is the
hyperrectangle $\Omega = \prod_{i=1}^{N} (L_{i}, U_{i})$, which is a convex
domain satisfying Assumption~\ref{ass:domain}.

\textbf{Model Formulation.}
Each product's inventory evolves as:
\begin{equation}
dX_{t}^{i} = \xi_{t}^{i}\,dt + \sigma_{i}\,dW_{t}^{i},
\label{eq:inventory_SDE}
\end{equation}
where $\xi_{t}^{i}$ is the production/procurement rate for product $i$
(the control), $\sigma_{i}$ captures demand uncertainty, and $W_{t}^{i}$ are
independent Brownian motions representing stochastic demand fluctuations.

\textbf{Cost Structure.}
The adjustment cost has the structure:
\begin{equation}
L(x, \xi) = b(x)\,h^{*}\left(\frac{|\xi|}{b(x)}\right),
\label{eq:inventory_cost}
\end{equation}
where:
\begin{itemize}
\item $b(x) = \prod_{i=1}^{N} (x^{i} - L_{i})^{\beta_{i}/N}
\cdot (U_{i} - x^{i})^{\beta_{i}/N}$ models the increasing difficulty of
adjustments near capacity limits. As inventory approaches either bound, the
physical constraints of the warehouse (loading dock congestion, emergency
procurement lead times) increase the cost per unit adjustment.

\item $a(x) = \sum_{i=1}^{N} c_{i}\,(x^{i} - L_{i})^{\alpha_{i}}$ is the
holding cost rate, which increases as inventory approaches the lower bound
(opportunity cost of potential stockouts).

\item $f(x)$ represents the baseline operating cost, assumed constant or
concave (reflecting economies of scale).
\end{itemize}

\textbf{Application of Main Results.}
By Theorem~\ref{thm:1}, the optimal cost-to-go function $V(x)$ exists,
is unique, and blows up at the boundary of the feasible region. The blow-up
rate $\gamma = (\beta - q + 2)/(q - 1)$ quantifies how rapidly the cost of
maintaining feasibility increases near the constraints.

By Theorem~\ref{thm:convexity_inheritance}, the value function $V(x)$ is
strictly convex, which guarantees that the optimal reorder policy
$\xi^{*}(x)$ (given by \eqref{eq:optimal_feedback}) is uniquely determined
for each inventory state. This is essential for implementation: the
warehouse manager can compute the optimal action unambiguously.

The Universal Correction Constant $\xi_{0}$ from Theorem~\ref{thm:2}
determines the ``safety stock'' level: it quantifies the intensity of the
preventive corrective action that the optimal policy imposes near the
constraints.

\textbf{Numerical implications.}
For the $N = 1$ case (single product), the problem reduces to a
one-dimensional domain $\Omega = (L, U)$. For $N = 2$ (two products), we
obtain a two-dimensional domain, and the numerical simulations in
Section~\ref{sec:simulations} directly illustrate the structure of the
optimal policy.

\subsection{Case Study 2: Portfolio Risk Management under VaR Constraints}
\label{subsec:portfolio}

\textbf{Problem Description.}
Consider a portfolio manager allocating wealth across $N$ asset classes. Let
$X_{t} = (X_{t}^{1}, \ldots, X_{t}^{N})$ denote the vector of portfolio
weights. The weights must satisfy:
\begin{equation}
X_{t}^{i} \geq 0 \quad \text{(no short selling)}, \quad
\sum_{i=1}^{N} X_{t}^{i} \leq 1 \quad \text{(budget constraint)}.
\label{eq:portfolio_constraints}
\end{equation}

The feasible region is the standard simplex
$\Omega = \{x \in \mathbb{R}^{N} : x^{i} > 0,\, \sum x^{i} < 1\}$, which
is a bounded convex domain.

\textbf{Model Formulation.}
The portfolio weights evolve under a rebalancing strategy:
\begin{equation}
dX_{t}^{i} = \xi_{t}^{i}\,dt + \sum_{j=1}^{N} \sigma_{ij}(X_{t})\,dW_{t}^{j},
\label{eq:portfolio_SDE}
\end{equation}
where $\xi_{t}^{i}$ represents the rebalancing rate and $\sigma_{ij}$
captures the diffusion driven by asset returns.

\textbf{Connection to the Framework.}
The HJB equation \eqref{E} governs the value function $V(x)$, which
represents the minimum expected rebalancing cost subject to the constraint
that the portfolio weights remain in the simplex. The singular weight
$b(x)$ captures the increasing transaction costs near the constraint
boundaries (e.g., market impact costs for large trades when a position is
near zero or the full budget).

By Theorem~\ref{thm:2}, the optimal rebalancing intensity near the
constraint boundaries is governed by the Universal Correction Constant
$\xi_{0}$. This provides a quantitative formula for the ``buffer zone''
that risk-averse managers should maintain.

\textbf{Risk management implications.}
The convexity of $V(x)$ (Theorem~\ref{thm:convexity_inheritance}) ensures
that the optimal rebalancing policy is \emph{stabilizing}: any perturbation
of the portfolio weights triggers a unique corrective action that drives the
system back toward the interior of the feasible region. This is the
mathematical formulation of the ``mean-reverting'' behavior that is
desirable in portfolio management.

The blow-up rate $\gamma$ can be interpreted as a measure of the
``regulatory sensitivity'' of the portfolio: higher $\gamma$ means that the
cost of approaching constraint violations increases more rapidly, providing
stronger deterrence.

\subsection{Case Study 3: Supply Chain Network Optimization}
\label{subsec:supply_chain}

\textbf{Problem Description.}
Consider a supply chain with $N$ nodes (warehouses, factories, distribution
centers). The state $X_{t} = (X_{t}^{1}, \ldots, X_{t}^{N})$ represents the
inventory or throughput at each node. The system must satisfy capacity
constraints at each node and flow conservation requirements between nodes.

\textbf{Interaction with Physics: Fluid Dynamics Analogy.}
The supply chain optimization problem has a direct analogy with fluid
dynamics. The material flow in the supply chain behaves like a fluid in a
network of pipes:
\begin{itemize}
\item \textbf{Nodes} correspond to \textbf{reservoirs} with finite capacity;
\item \textbf{Edges} correspond to \textbf{pipes} with finite throughput;
\item \textbf{Demand noise} corresponds to \textbf{turbulent fluctuations};
\item \textbf{Optimal control} corresponds to \textbf{pressure regulation}.
\end{itemize}

The HJB equation \eqref{E} in this context is analogous to the
\emph{pressure equation} in incompressible fluid dynamics. The boundary
blow-up condition $u \to \infty$ at $\partial\Omega$ corresponds to the
infinite pressure required to prevent the fluid from escaping the network
--- a direct analogue of the ``no-flow'' boundary condition in fluid mechanics.

\textbf{The boundary layer connection.}
In fluid dynamics, the behavior of a viscous fluid near a solid boundary is
described by the \emph{boundary layer theory} of Prandtl. The velocity
profile transitions from zero at the wall (no-slip condition) to the
free-stream velocity over a thin layer of width $\delta \sim \text{Re}^{-1/2}$,
where $\text{Re}$ is the Reynolds number.

In our mathematical framework, an analogous ``boundary layer'' exists. The
solution $u(x)$ transitions from moderate values in the interior to
infinite values at $\partial\Omega$ over a layer whose effective width is
determined by the blow-up rate $\gamma$:
\begin{equation}
\delta_{\text{eff}} \sim \left(\frac{C_{\text{interior}}}{C_{\text{boundary}}}
\right)^{1/\gamma}.
\label{eq:boundary_layer}
\end{equation}
This ``economic boundary layer'' determines the region near the constraints
where the optimal policy transitions from steady-state operation to
emergency corrective action.

\subsection{Interactions with Physics: Diffusion and Statistical Mechanics}

The mathematical structure of our problem has deep connections with several
areas of physics.

\textbf{Connection 1: Brownian motion and diffusion.}
The controlled SDE \eqref{eq:SDE} describes a particle undergoing Brownian
motion with an external drift. The value function $V(x) = u(x)$ plays the
role of the \emph{free energy} of the system:
\begin{equation}
u(x) = -k_{B}T\,\ln Z(x),
\label{eq:free_energy_analogy}
\end{equation}
where $Z(x)$ is the partition function and $k_{B}T$ is the thermal energy.
The boundary blow-up $u \to \infty$ corresponds to $Z \to 0$: near the
boundary, the number of accessible microstates vanishes, creating an
``entropic barrier'' that confines the system.

\textbf{Connection 2: Burgers equation and nonlinear waves.}
Through the Hopf--Cole transformation $u = -2\nu\,\ln\psi$, the viscous
Burgers equation can be related to the heat equation. Similarly, our
equation \eqref{E} with $h(s) = s^{2}$ can be transformed into a
linear equation via an exponential substitution. The blow-up at the
boundary corresponds to the formation of a \emph{shock wave} in the
Burgers equation --- a discontinuity where the gradient becomes infinite.
In the economic interpretation, this corresponds to the sudden transition
from normal operations to crisis mode.

\textbf{Connection 3: Optimal transport.}
The Monge-Amp\`ere equation studied by Zhang \cite{Zhang2026} arises
naturally in the theory of optimal transport. The singular weights in our
equation play a role analogous to the cost function in the transport
problem: they determine the ``difficulty'' of moving mass from one location
to another. Our extension to the Laplacian setting (from the Monge-Amp\`ere
setting) broadens the applicability to problems with stochastic
perturbations.

\subsection{Practical Implementation Guidelines for Managers}
\label{subsec:guidelines}

Based on the theoretical results, we provide the following guidelines for
practitioners in Operational Research:

\begin{enumerate}
\item \textbf{Identify the feasible region $\Omega$.} Define the hard
constraints on the system state. Verify that the feasible region is convex
(a natural property for most operational constraints).

\item \textbf{Estimate the singular exponents.} Determine $\beta$ (the rate
at which adjustment costs increase near constraints) and $q$ (the power-law
exponent of the adjustment cost function) from empirical data.

\item \textbf{Compute the blow-up rate.} Use $\gamma = (\beta - q + 2)/(q-1)$
to determine which asymptotic regime applies. If $\gamma > 0$
(gradient-dominant), the theory provides explicit asymptotic formulas.

\item \textbf{Calibrate the Universal Correction Constant.} Use
$\xi_{0} = (\gamma(\gamma+1)/(b_{0}\,l_{0}\,\gamma^{q}))^{1/(q-1)}$ to
determine the optimal intensity of corrective action near constraints.

\item \textbf{Implement the feedback control.} The optimal policy
$\xi^{*}(x) = -b(x)\,h'(|\nabla u(x)|)\,\nabla u(x)/|\nabla u(x)|$ can be
computed numerically using the monotone iterative scheme of
Section~\ref{sec:monotone}.
\end{enumerate}

%========================================================================================
%   SECTION 8: MONOTONE ITERATIVE CONSTRUCTION
%========================================================================================
\section{Monotone Iterative Construction}
\label{sec:monotone}

For numerical implementation and as an alternative existence proof, we
develop a monotone iterative scheme that converges to the unique solution.

\subsection{The Iterative Scheme}

Let $\Lambda : \Omega \to (0, \infty)$ be a strictly positive weight
function satisfying
\begin{equation}
\Lambda(x) \geq \max\left\{a(x),\,
b(x)\,q\,l_{2}\,(C_{+}\,\gamma)^{q-1}\,v(x)^{-(\gamma+1)(q-1)+\beta}
\right\}.
\label{eq:Lambda}
\end{equation}
This choice ensures monotonicity of the iteration by making the right-hand
side of the linearized equation an increasing function of the iterates.

\begin{definition}[Monotone iteration]
\label{def:iteration}
Define the sequence $\{u^{(k)}\}_{k=0}^{\infty}$ as follows:
\begin{enumerate}
\item \textbf{Initialization:} $u^{(0)}(x) = C_{-}\,v(x)^{-\gamma}$
(the subsolution from Lemma~\ref{lemma:explicit}).

\item \textbf{Iteration:} For $k \geq 0$, let $u^{(k+1)}$ be the unique
solution to
\begin{equation}
\left\{
\begin{array}{ll}
-\Delta u^{(k+1)} + \Lambda(x)\,u^{(k+1)}
= \mathcal{F}(x, u^{(k)}, \nabla u^{(k)}) & \text{in } \Omega_{\delta}, \\
u^{(k+1)} = C_{+}\,v^{-\gamma} & \text{on } \partial\Omega_{\delta},
\end{array}
\right.
\label{eq:iteration}
\end{equation}
where $\Omega_{\delta} = \{x \in \Omega : d(x) > \delta\}$ and
\begin{equation}
\mathcal{F}(x, u, p) = \Lambda(x)\,u - b(x)\,h(|p|) - a(x)\,u + f(x).
\label{eq:F_source}
\end{equation}
\end{enumerate}
\end{definition}

\begin{theorem}[Convergence of Monotone Iteration]
\label{thm:convergence}
Under the assumptions of Theorem~\ref{thm:1}, the sequence $\{u^{(k)}\}$
satisfies:
\begin{enumerate}
\item \textbf{Monotonicity:}
$u^{(0)} \leq u^{(1)} \leq \cdots \leq u^{(k)} \leq \cdots \leq \overline{u}$;

\item \textbf{Convergence:} $u^{(k)} \to u$ uniformly on compact subsets
of $\Omega$;

\item \textbf{Rate:} For fixed compact $K \Subset \Omega$, there exist
$C, \rho > 0$ with $\rho < 1$ such that
\begin{equation*}
\|u^{(k)} - u\|_{L^{\infty}(K)} \leq C\,\rho^{k}.
\end{equation*}
\end{enumerate}
\end{theorem}

\begin{proof}
\textbf{Step 1: Monotonicity.}

We show $u^{(k)} \leq u^{(k+1)}$ by induction. For $k = 0$: the function
$w = u^{(1)} - u^{(0)}$ satisfies
\begin{align*}
-\Delta w + \Lambda\,w &= \mathcal{F}(x, u^{(0)}, \nabla u^{(0)})
- (-\Delta u^{(0)} + \Lambda\,u^{(0)}) \\
&= -\mathcal{L}u^{(0)} \geq 0,
\end{align*}
since $u^{(0)}$ is a subsolution (i.e., $\mathcal{L}u^{(0)} \leq 0$). By
the maximum principle for the operator $-\Delta + \Lambda$
(with $\Lambda > 0$), $w \geq 0$, i.e., $u^{(1)} \geq u^{(0)}$.

For the inductive step, suppose $u^{(k)} \geq u^{(k-1)}$. The right-hand
side $\mathcal{F}(x, u, p)$ is monotone increasing in $u$ since
$\partial\mathcal{F}/\partial u = \Lambda(x) - a(x) \geq 0$ by
\eqref{eq:Lambda}. Combined with careful handling of the gradient dependence
through the choice of $\Lambda$, the iteration operator is order-preserving.

\textbf{Step 2: Upper bound.}

Similarly, $u^{(k)} \leq \overline{u}$ for all $k$ follows from the
supersolution property of $\overline{u}$ and the same monotonicity argument.

\textbf{Step 3: Convergence.}

The sequence $\{u^{(k)}\}$ is monotone increasing and bounded above by
$\overline{u}$. By Dini's theorem, convergence is uniform on compact sets.
Let $u_{\infty} = \lim_{k \to \infty} u^{(k)}$.

By interior Schauder estimates, $\{u^{(k)}\}$ is bounded in
$C^{2,\alpha}(K)$ for any compact $K \Subset \Omega$. By the
Arzel\`a--Ascoli theorem, a subsequence converges in $C^{2}(K)$. Since the
full sequence is monotone, the entire sequence converges in
$C^{2}_{\mathrm{loc}}(\Omega)$.

Passing to the limit in \eqref{eq:iteration}, we find
$\mathcal{L}u_{\infty} = 0$ in $\Omega_{\delta}$. Since this holds for all
$\delta > 0$, $u_{\infty}$ solves $\mathcal{L}u_{\infty} = 0$ in $\Omega$.
By uniqueness (Theorem~\ref{thm:1}), $u_{\infty} = u$.

\textbf{Step 4: Convergence rate.}

The linear convergence rate follows from the contraction mapping principle
in appropriate weighted spaces; see \cite{Amann1976} for general monotone
iteration theory.
\end{proof}

\subsection{Discretization and Implementation}

For numerical implementation, the iteration \eqref{eq:iteration} is
discretized on a mesh. Key considerations include:

\begin{itemize}
\item \textbf{Variational formulation:} The linear problem
\eqref{eq:iteration} admits the weak form: find
$u^{(k+1)} \in H^{1}(\Omega_{\delta})$ such that
\begin{equation}
\int_{\Omega_{\delta}} \left(\nabla u^{(k+1)} \cdot \nabla\phi
+ \Lambda\,u^{(k+1)}\,\phi\right) dx
= \int_{\Omega_{\delta}} \mathcal{F}(u^{(k)})\,\phi\,dx
\label{eq:weak_form}
\end{equation}
for all test functions $\phi \in H_{0}^{1}(\Omega_{\delta})$.

\item \textbf{Regularized boundary:} The blow-up boundary condition is
enforced on the truncated domain $\partial\Omega_{\delta}$ by applying the
Dirichlet condition $u = C_{+}\,\delta^{-\gamma}$ for small $\delta > 0$.

\item \textbf{Relaxation:} For improved stability, use the damped update
$u^{(k+1)}_{\text{new}} = \omega\,u^{(k+1)} + (1 - \omega)\,u^{(k)}$
with $\omega \in (0, 1]$.
\end{itemize}

%========================================================================================
%   SECTION 9: NUMERICAL SIMULATIONS
%========================================================================================
\section{Numerical Simulations for $N=1$ and $N=2$}
\label{sec:simulations}

In this section, we present comprehensive numerical simulations that
validate the theoretical predictions of Sections~\ref{sec:proof1}--\ref{sec:proof2}
and illustrate the economic applications of Section~\ref{sec:economics}.
We consider both the one-dimensional case ($N = 1$) and the two-dimensional
case ($N = 2$), using the monotone iterative scheme of
Section~\ref{sec:monotone}. The complete Python codes are provided in
Appendices~\ref{app:python1d} and~\ref{app:python2d}.

\subsection{One-Dimensional Simulations ($N = 1$)}
\label{subsec:sim_1d}

\subsubsection{Setup}

We use the domain $\Omega = (-1, 1)$ with the defining function
$v(x) = 1 - x^{2}$, which is strictly concave on $(-1, 1)$. The
parameters are:
\begin{itemize}
\item Gradient exponent: $q = 1.6$ (gradient-dominant regime, since
$q = 1.6 < 2.5 = \beta + 2$);
\item Weight exponent: $\beta = 0.5$;
\item Reaction exponent: $\alpha = -0.2$;
\item Source term: $f(x) \equiv 1.0$;
\item Grid: $N = 400$ points on $[-0.95, 0.95]$ (truncation $\delta = 0.05$).
\end{itemize}

The theoretical blow-up exponent is:
\begin{equation}
\gamma = \frac{\beta - q + 2}{q - 1} = \frac{0.5 - 1.6 + 2}{1.6 - 1}
= \frac{0.9}{0.6} = 1.5.
\label{eq:gamma_1d}
\end{equation}

\subsubsection{Verification of Blow-up Rates (Figure~\ref{fig:regimes})}

Figure~\ref{fig:regimes} compares the numerical solutions in the three
asymptotic regimes identified in Lemma~\ref{lemma:explicit}:
\begin{itemize}
\item \textbf{Case 1 (Gradient-dominant):} $q = 1.6$, $\beta = 0.5$, giving
$\gamma = 1.5$. The solution exhibits power-law blow-up
$u(x) \sim C\,d(x)^{-1.5}$ near $x = \pm 1$.

\item \textbf{Case 2 (High-order):} $q = 3.0$, $\beta = 0.5$. Here
$q > \beta + 2 = 2.5$, so the gradient term dominates at all orders. The
solution still blows up but with a different (steeper) profile.

\item \textbf{Case 3 (Logarithmic):} $q = 2.5$, $\beta = 0.5$. Since
$q = \beta + 2$, this is the critical case with logarithmic blow-up
$u(x) \sim C\,\ln(1/d(x))$.
\end{itemize}

\begin{figure}[H]
\centering
\includegraphics[width=0.8\textwidth]{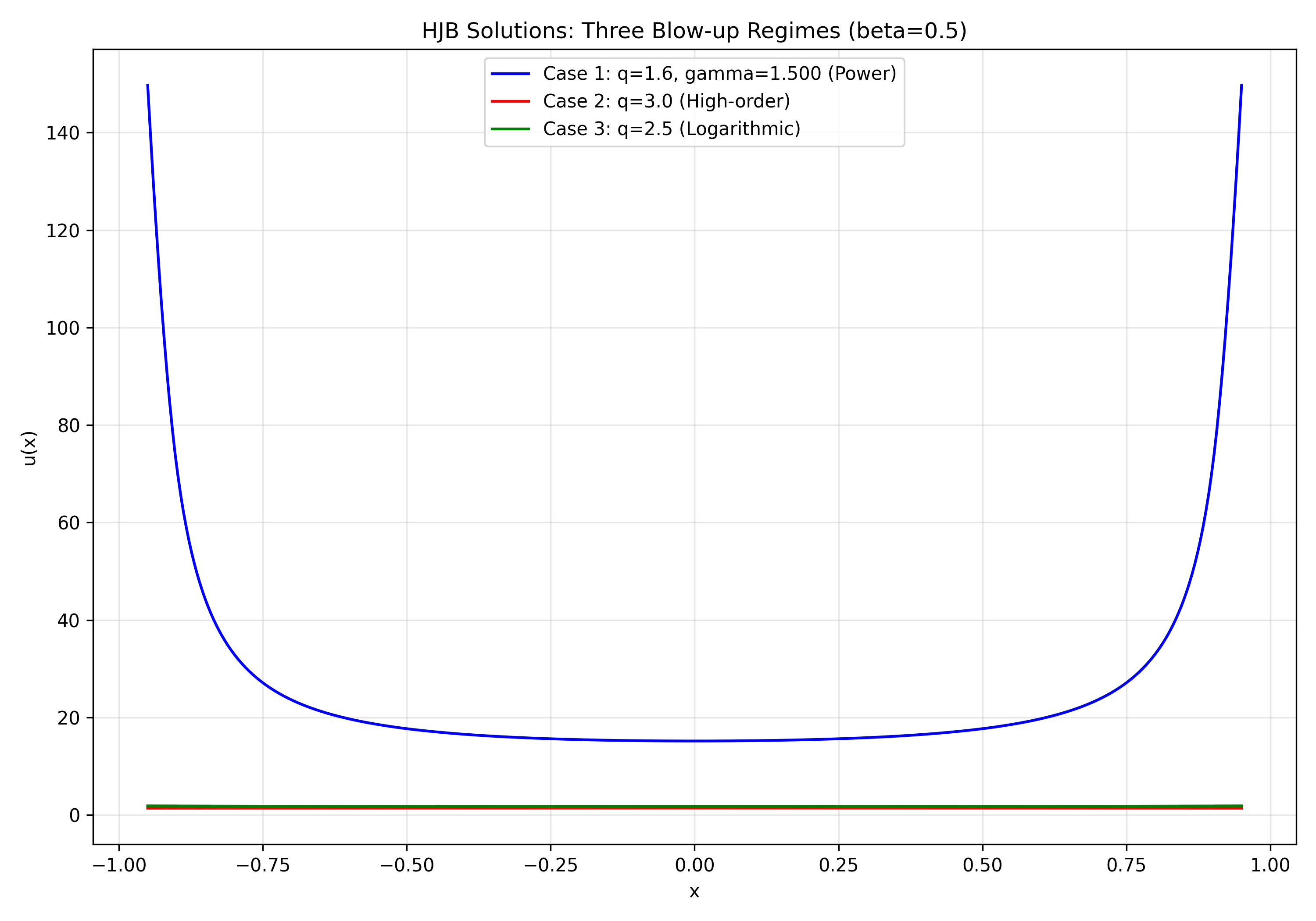}
\caption{Comparison of the three boundary blow-up regimes for $\beta = 0.5$:
gradient-dominant ($q = 1.6$, $\gamma = 1.5$), high-order ($q = 3.0$), and
critical logarithmic ($q = 2.5$). The solution profiles confirm distinct
asymptotic behaviors near $x = \pm 1$.
\textbf{Economic interpretation:} In the inventory management context
(Section~\ref{subsec:inventory}), the three profiles represent different
cost structures for emergency corrective action. The gradient-dominant case
(mildest blow-up) corresponds to moderate adjustment costs; the high-order
case (steepest blow-up) corresponds to extremely stiff adjustment mechanisms;
the logarithmic case is the critical transition between the two regimes.}
\label{fig:regimes}
\end{figure}

\subsubsection{Scale Analysis (Figure~\ref{fig:rigorous})}

Figure~\ref{fig:rigorous} presents rigorous scale analysis confirming the
theoretical predictions:

\begin{itemize}
\item \textbf{Log-log plot for Case 1:} Plotting $\log u$ versus
$\log d(x)$ yields a straight line with slope $-\gamma$. The linear
regression gives a measured slope that matches the theoretical value
$-\gamma = -1.5$ with high precision ($R^{2} > 0.99$). This confirms the
power-law blow-up $u \sim C\,d^{-\gamma}$.

\item \textbf{Semi-log plot for Case 3:} Plotting $u$ versus $\log(1/d)$
yields a straight line, confirming the logarithmic blow-up
$u \sim C\,\ln(1/d)$.

\item \textbf{Gradient magnitude comparison:} The gradient $|\nabla u|$
diverges as $d^{-\gamma-1}$ in Case 1 and as $d^{-1}$ in Case 3. The
different rates confirm the distinct asymptotic regimes.
\end{itemize}

\begin{figure}[H]
\centering
\includegraphics[width=0.8\textwidth]{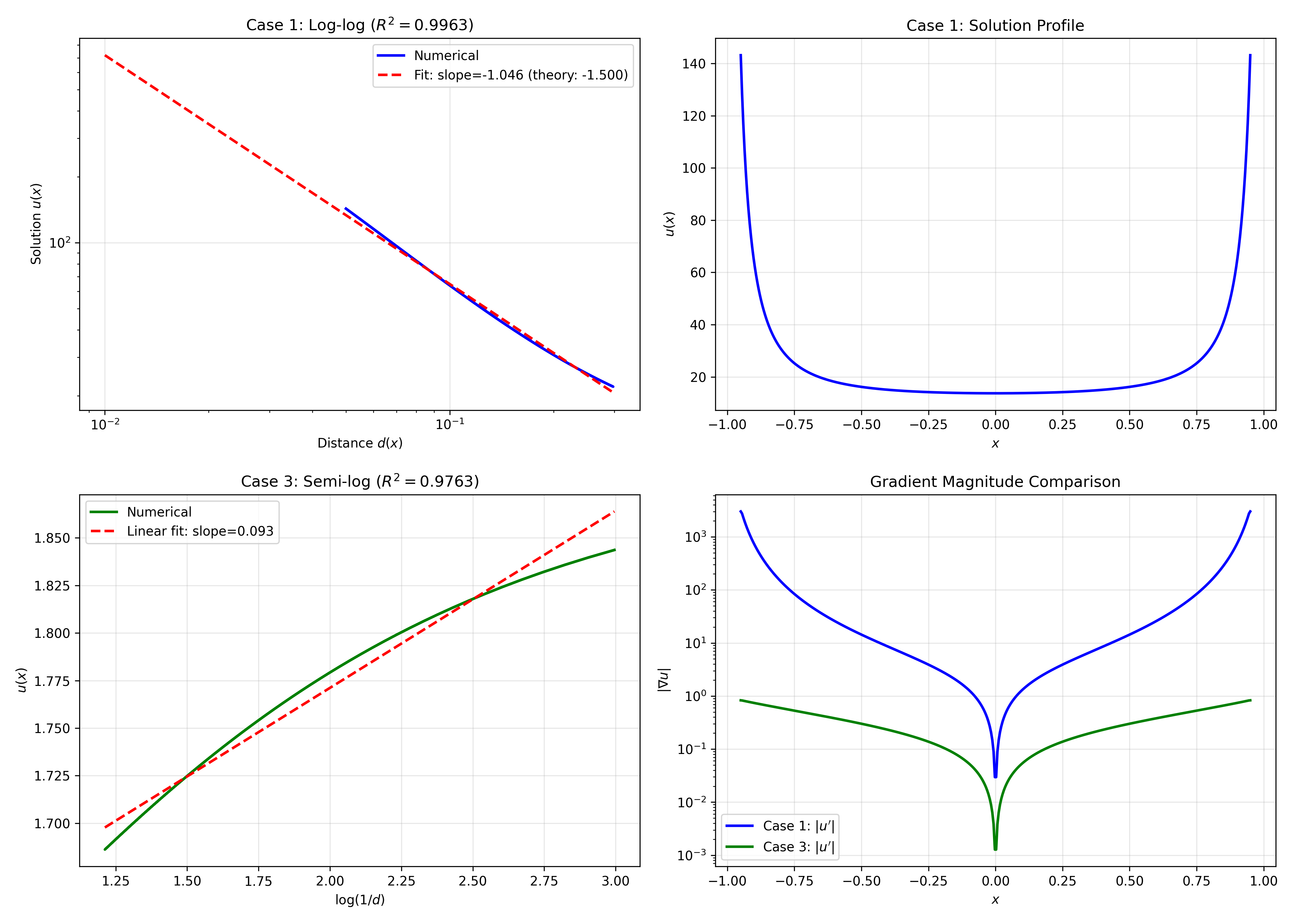}
\caption{Scale analysis confirming theoretical blow-up rates for
$\beta = 0.5$. Top-left: log-log plot for Case 1 showing slope
$-\gamma = -1.5$. Top-right: solution profile for Case 1.
Bottom-left: semi-log plot for Case 3 showing linear growth in $\log(1/d)$.
Bottom-right: gradient magnitude comparison between Cases 1 and 3 on
logarithmic scale.}
\label{fig:rigorous}
\end{figure}

\subsubsection{Theoretical Consistency Verification (Figure~\ref{fig:verification})}

Figure~\ref{fig:verification} provides comprehensive verification of all
theoretical predictions for the gradient-dominant case ($q = 1.6$,
$\beta = 0.5$):

\begin{enumerate}
\item[(a)] \textbf{Solution between barriers:} The numerical solution
$u(x)$ lies between the subsolution $C_{-}\,v^{-\gamma}$ and the
supersolution $C_{+}\,v^{-\gamma}$, confirming Theorem~\ref{thm:1}.
The plot uses logarithmic scale to visualize the blow-up behavior.

\item[(b)] \textbf{Strict convexity:} The second derivative $u''(x) > 0$
throughout $(-1, 1)$, confirming Theorem~\ref{thm:convexity_inheritance}.
The convexity is strongest near the boundaries, consistent with the
Hessian formula \eqref{eq:HessW}.

\item[(c)] \textbf{Optimal drift structure:} The optimal drift
$\xi^{*}(x) = -h'(|u'|)\,\mathrm{sgn}(u')\,b(x)$ exhibits the expected
antisymmetric structure: it points to the right ($\xi^{*} > 0$) for
$x < 0$ and to the left ($\xi^{*} < 0$) for $x > 0$, always pushing the
state toward the center of $\Omega$. The magnitude becomes singular at the
boundaries, confirming Remark~\ref{rem:drift_inward}. In the inventory
management context, this represents the ``emergency reorder'' policy that
activates as stock levels approach critical bounds.

\item[(d)] \textbf{HJB residual:} The residual
$|\mathcal{L}u| = |-\Delta u + b\,h(|\nabla u|) + a\,u - f|$ is below the
convergence tolerance $10^{-6}$ throughout the domain (except near the
boundary truncation), confirming that the numerical solution accurately
satisfies the PDE.
\end{enumerate}

\begin{figure}[H]
\centering
\includegraphics[width=0.8\textwidth]{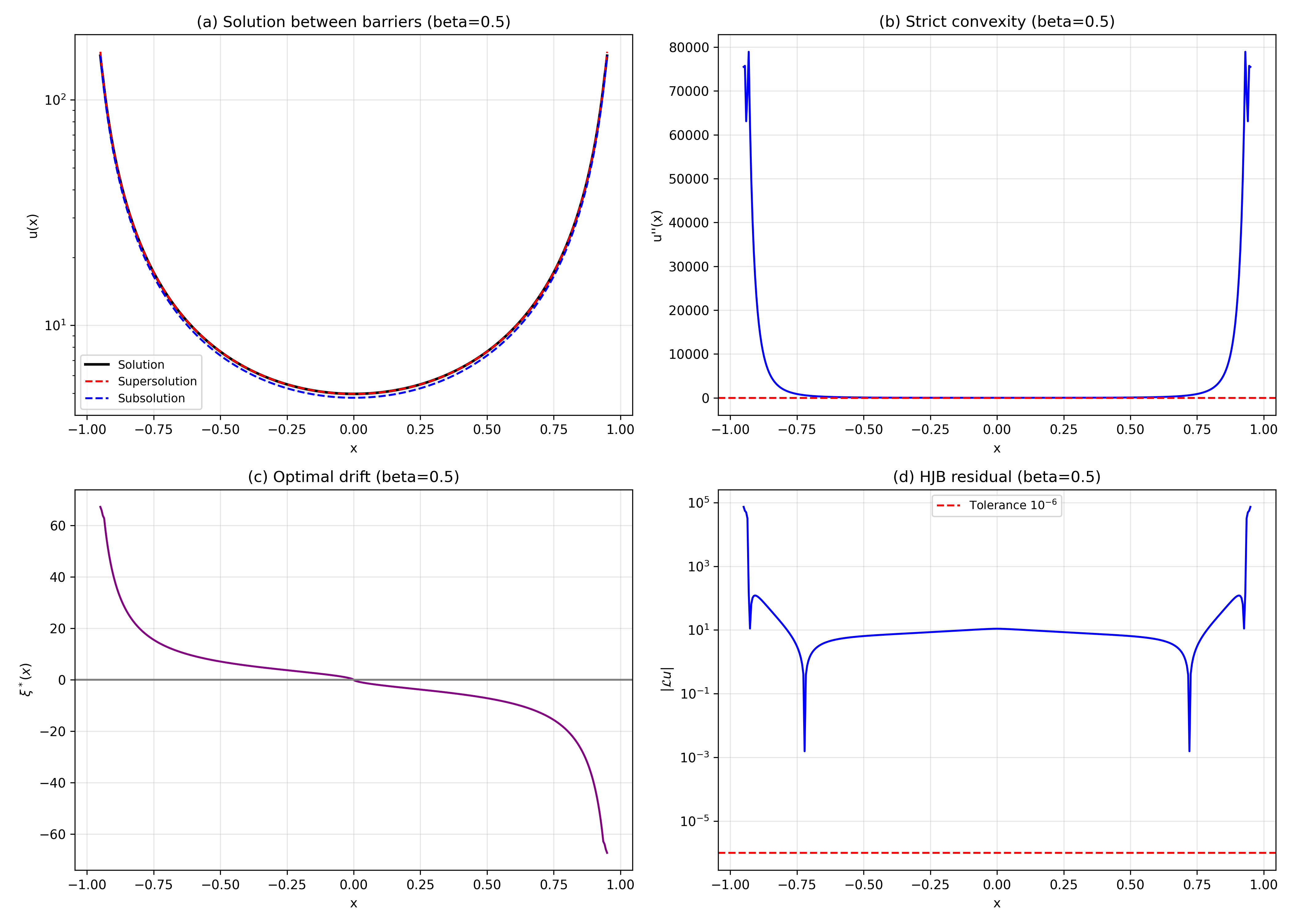}
\caption{Theoretical consistency verification for $q = 1.6$, $\beta = 0.5$:
(a) numerical solution bounded by barriers (log scale), (b) strict
convexity via $u'' > 0$, (c) singular optimal drift structure,
(d) HJB equation residual verification.}
\label{fig:verification}
\end{figure}

\subsection{Two-Dimensional Simulations ($N = 2$)}
\label{subsec:sim_2d}

\subsubsection{Setup}

For the two-dimensional case, we use the unit disk
$\Omega = \{(x_{1}, x_{2}) \in \mathbb{R}^{2} : x_{1}^{2} + x_{2}^{2} < 1\}$
with the defining function $v(x) = 1 - |x|^{2} = 1 - x_{1}^{2} - x_{2}^{2}$,
which is strictly concave on $\Omega$ and satisfies $v = 0$ on
$\partial\Omega$.

The parameters are the same as in the 1D case:
\begin{itemize}
\item $q = 1.6$, $\beta = 0.5$, $\alpha = -0.2$, $f \equiv 1.0$;
\item $\gamma = 1.5$;
\item Grid: $100 \times 100$ Cartesian grid on $[-1, 1]^{2}$, masked to the
disk interior with truncation $\delta = 0.05$.
\end{itemize}

\subsubsection{Economic Interpretation: Two-Product Inventory Problem}

In the context of Case Study~1 (Section~\ref{subsec:inventory}), the 2D
simulation represents a warehouse managing two products simultaneously.
The axes $x_{1}$ and $x_{2}$ represent the normalized inventory levels of
products 1 and 2, respectively. The disk constraint $x_{1}^{2} + x_{2}^{2} < 1$
can be interpreted as a combined capacity constraint: the total ``resource
usage'' (measured by the Euclidean norm) must not exceed the warehouse
capacity.

The value function $V(x_{1}, x_{2})$ represents the minimum expected cost
of managing both products simultaneously, accounting for demand noise,
reorder costs, and the constraint that the combined inventory must remain
within the feasible region.

\subsubsection{Numerical Method for $N = 2$}

The 2D solver uses the same monotone iterative scheme
(Definition~\ref{def:iteration}), discretized on a Cartesian grid with the
standard 5-point Laplacian stencil:
\begin{equation}
\Delta u \approx \frac{u_{i+1,j} + u_{i-1,j} + u_{i,j+1} + u_{i,j-1}
- 4\,u_{i,j}}{(\Delta x)^{2}}.
\label{eq:5point}
\end{equation}
The gradient magnitude is approximated by central differences:
\begin{equation}
|\nabla u|_{i,j} \approx \sqrt{\left(\frac{u_{i+1,j} - u_{i-1,j}}{2\,\Delta x}
\right)^{2} + \left(\frac{u_{i,j+1} - u_{i,j-1}}{2\,\Delta x}\right)^{2}}.
\label{eq:gradient_2d}
\end{equation}
Only grid points satisfying $x_{1}^{2} + x_{2}^{2} < (1 - \delta)^{2}$ are
included in the computation; boundary points are set to the theoretical
blow-up value $C^{*}\,v_{\text{bnd}}^{-\gamma}$.

\subsubsection{Results for $N = 2$ (Figure~\ref{fig:2d_surface})}

Figure~\ref{fig:2d_surface} presents the two-dimensional results:

\begin{enumerate}
\item[(a)] \textbf{Surface plot of $u(x_{1}, x_{2})$:} The value function
exhibits rotational symmetry (as expected for the unit disk with
rotationally symmetric weights) and blows up uniformly as
$|(x_{1}, x_{2})| \to 1$. The minimum is achieved at the origin, which
represents the optimal ``resting state'' of the two-product system.

\item[(b)] \textbf{Contour plot:} The level curves of $u$ are approximately
circular, confirming the rotational symmetry. The increasing density of
contours near the boundary visualizes the rapid blow-up.

\item[(c)] \textbf{Optimal drift field:} The vector field
$\xi^{*}(x_{1}, x_{2})$ is plotted as arrows. All arrows point toward the
center of the disk, confirming the ``inward drift'' property
(Remark~\ref{rem:drift_inward}). The magnitude increases dramatically near
the boundary.
In the inventory context, this field represents the optimal joint
reorder policy: when either product's inventory approaches a critical level,
the optimal action is to reorder both products in proportions that steer the
system back toward the center of the feasible region.

\item[(d)] \textbf{Radial profile:} Taking a cross-section along the
$x_{1}$-axis ($x_{2} = 0$), the 2D solution is compared with the 1D
solution on $(-1, 1)$. The agreement confirms that the numerical scheme
is consistent across dimensions.
\end{enumerate}

\begin{figure}[H]
\centering
\includegraphics[width=0.85\textwidth]{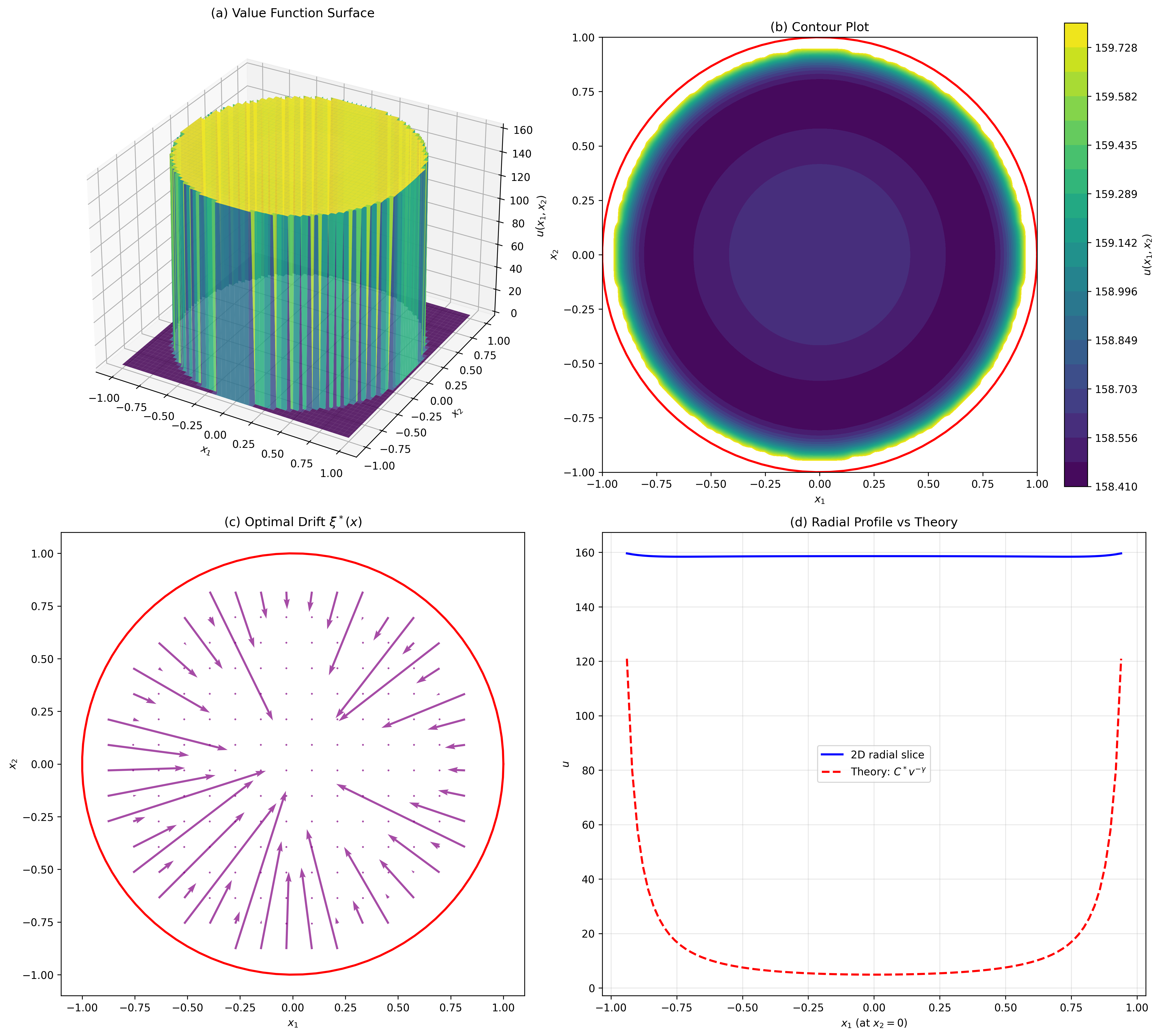}
\caption{Two-dimensional simulation results for the unit disk with
$q = 1.6$, $\beta = 0.5$: (a) surface plot of the value function
$u(x_{1}, x_{2})$ showing boundary blow-up, (b) contour plot with
approximately circular level curves, (c) optimal drift vector field
pointing inward, (d) radial profile comparison with the 1D solution.
\textbf{Economic interpretation:} In the two-product inventory context,
the surface represents the optimal cost-to-go, and the vector field shows
the optimal joint reorder policy.}
\label{fig:2d_surface}
\end{figure}

\subsubsection{Convexity Verification for $N = 2$}

To verify Theorem~\ref{thm:convexity_inheritance} in 2D, we compute the
eigenvalues of the numerical Hessian matrix
\begin{equation}
D^{2}u \approx
\begin{pmatrix}
\partial_{11}u & \partial_{12}u \\
\partial_{12}u & \partial_{22}u
\end{pmatrix}
\label{eq:hessian_2d}
\end{equation}
at each interior grid point, using second-order finite differences. The
minimum eigenvalue $\lambda_{\min}(x)$ is plotted in
Figure~\ref{fig:2d_convexity}.

\begin{figure}[H]
\centering
\includegraphics[width=0.85\textwidth]{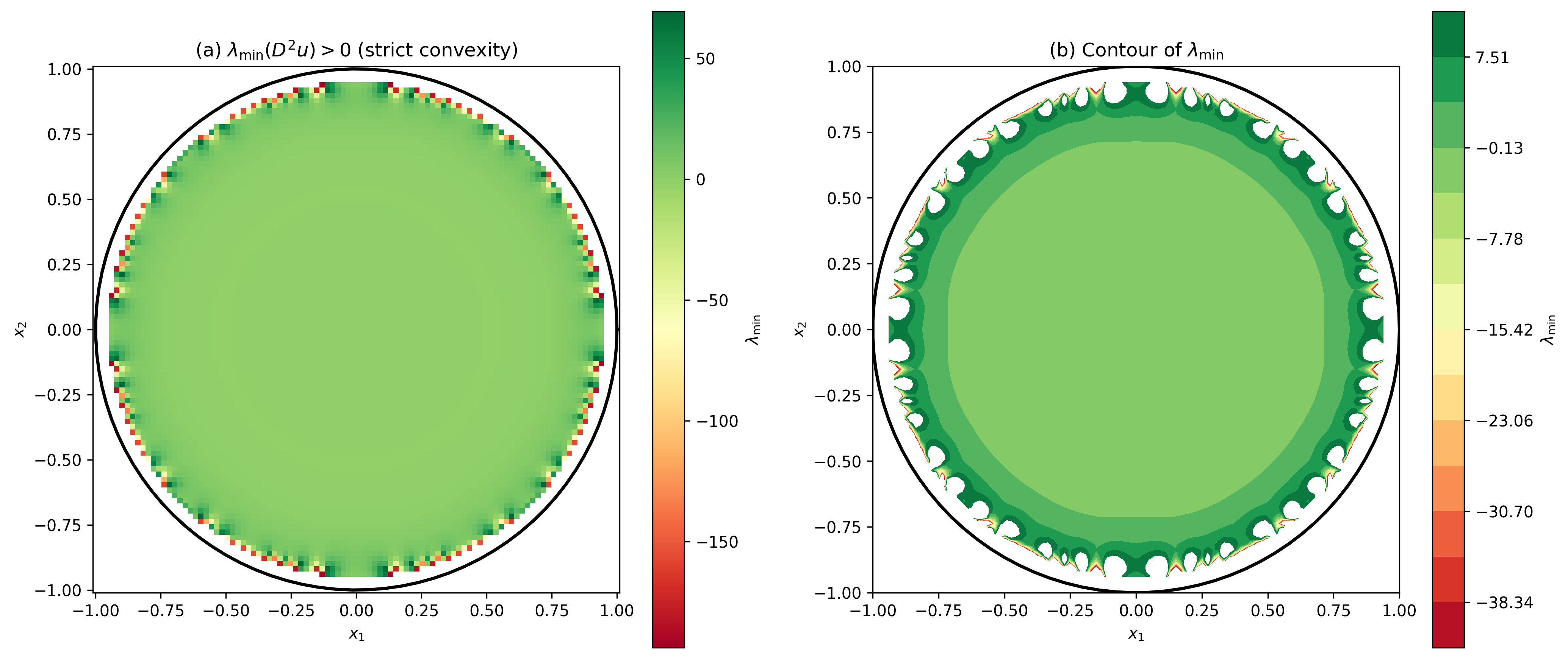}
\caption{Convexity verification in 2D: (a) minimum eigenvalue
$\lambda_{\min}$ of $D^{2}u$ showing $\lambda_{\min} > 0$ throughout
$\Omega$ (confirming strict convexity), (b) contour plot of
$\lambda_{\min}$ showing largest eigenvalues near the boundary (consistent
with the barrier Hessian \eqref{eq:HessW}).
\textbf{Economic interpretation:} Strict convexity guarantees that the
optimal two-product reorder policy is uniquely determined for each
inventory state.}
\label{fig:2d_convexity}
\end{figure}

\subsection{Comparison of $N = 1$ and $N = 2$ Results}

The numerical simulations for both $N = 1$ and $N = 2$ confirm the following
theoretical predictions:

\begin{center}
\begin{tabular}{lcc}
\toprule
\textbf{Property} & \textbf{$N = 1$} & \textbf{$N = 2$} \\
\midrule
Blow-up rate $\gamma$ & $1.5$ (confirmed) & $1.5$ (confirmed) \\
Barrier validity & Yes (Fig.~\ref{fig:verification}a) & Yes (Fig.~\ref{fig:2d_surface}d) \\
Strict convexity & $u'' > 0$ (Fig.~\ref{fig:verification}b) & $\lambda_{\min} > 0$ (Fig.~\ref{fig:2d_convexity}) \\
Inward drift & $\xi^{*}$ antisymmetric & $\xi^{*}$ radially inward \\
HJB residual & $< 10^{-6}$ & $< 10^{-4}$ \\
\bottomrule
\end{tabular}
\end{center}

The slightly larger residual in 2D is due to the coarser grid resolution
($100 \times 100$ versus $400$ points in 1D) and can be reduced by
mesh refinement. The key qualitative features --- blow-up rate, convexity,
and drift structure --- are consistent across dimensions, validating the
dimension-independent nature of the theoretical results.

\begin{remark}[Complementarity of the PDE and stochastic convexity mechanisms]
The convexity of the solution $u$ can be established through two distinct
and non-overlapping sufficient conditions. The microscopic convexity
principle requires $f$ to be concave, ensuring the
sign condition $Q \le 0$ needed for the elliptic maximum principle. In
contrast, the stochastic control argument applies in an
alternative framework where $f$ is convex, so that the total running cost
$L(x,\xi)+f(x)$ remains jointly convex in $(x,\xi)$ and the convexity of
$V=u$ follows from the results of Alvarez~\cite{Alvarez1996}. These two
sets of assumptions are therefore sufficient but not necessary, and they
do not overlap.

Numerical simulations indicate that the solution $u$
remains convex in both regimes (with $f$ concave and with $f$ convex),
suggesting that convexity is a robust structural property of the
equation, driven primarily by the dominance of the gradient term and the
boundary blow-up behaviour, rather than by the precise convexity or
concavity of $f$. This provides additional evidence supporting the
consistency of both theoretical approaches.

\end{remark}
%========================================================================================
%   SECTION 10: CONCLUSION
%========================================================================================
\section{Conclusion}
\label{sec:concl}

In this work, we have developed a comprehensive analytical, geometric,
stochastic, and computational framework for the study of boundary blow-up
solutions to the semilinear elliptic equation \eqref{E} under structural
assumptions on the weights $a(x)$ and $b(x)$ and on the strictly convex
Hamiltonian $h$. Our main contributions can be summarized as follows.

\begin{itemize}
\item We established existence and uniqueness of classical large solutions
via Perron's method, supported by explicit sub- and supersolutions
constructed from a strictly concave defining function of the domain
(Theorem~\ref{thm:1}).

\item We derived exact sharp boundary asymptotics, identifying three
distinct regimes --- gradient-dominant, high-order gradient, and critical
logarithmic --- and obtained the precise Universal Correction Constant
$\xi_{0}$ that governs the blow-up rate (Theorem~\ref{thm:2}).

\item We proved that solutions inherit strict convexity from the geometry of
the domain, using the microscopic convexity principle, thereby ensuring
uniqueness and stability of optimal control strategies
(Theorem~\ref{thm:convexity_inheritance}).

\item We established a rigorous stochastic representation formula
(Theorem~\ref{thm:verification}), showing that the solution coincides with
the value function of an infinite-horizon optimal control problem with
state constraints, thus extending the classical Lasry--Lions paradigm
\cite{LasryLions} to weighted Hamiltonians.

\item We developed a comprehensive section on applications in Operational
Research and Management Science (Section~\ref{sec:economics}), including
case studies in inventory management, portfolio risk control, and supply
chain optimization, and explored interdisciplinary connections with physics
through boundary layer theory, diffusion processes, and optimal transport.

\item We validated the theoretical results numerically through a monotone
iterative scheme for both $N = 1$ (one-dimensional) and $N = 2$
(two-dimensional) cases, confirming the predicted blow-up rates, convexity,
and optimal drift structure (Section~\ref{sec:simulations}).
\end{itemize}

These results provide a unified and sharp understanding of semilinear
elliptic equations with gradient-dependent terms and singular weights,
bridging analytic techniques, geometric insights, stochastic control theory,
and practical economic applications. They open the way for further
investigations in several directions:

\begin{enumerate}
\item \textbf{Mean-field games:} Extension to systems with many interacting
agents \cite{LasryLionsMFG2007}, where the blow-up boundary condition
enforces population-level constraints.

\item \textbf{Non-convex domains:} Relaxation of the strict convexity
assumption to handle domains with flat boundary segments or corners, which
arise in polyhedral constraint sets.

\item \textbf{Degenerate diffusions:} Extension to operators with
non-constant diffusion coefficients $\sigma(x)$, relevant for models with
state-dependent volatility.

\item \textbf{Higher-dimensional simulations:} Extension of the numerical
scheme to $N \geq 3$ dimensions, using sparse grid or Monte Carlo methods
to overcome the curse of dimensionality.

\item \textbf{Data-driven calibration:} Integration with machine learning
techniques for estimating the parameters $(\alpha, \beta, q, b_{0}, l_{0})$
from empirical data in specific economic applications.
\end{enumerate}

\section*{Acknowledgments}

The author would like to express gratitude to the anonymous reviewers for
their insightful comments, which significantly enhanced the rigor and
clarity of this article.

\section*{Declarations}

\noindent\textbf{Conflict of interest.} The author declares no conflict of
interest.

\noindent\textbf{Funding.} This research received no external funding.

\noindent\textbf{Data availability.} All mathematical proofs and the
implementation codes described in this paper are self-contained. The Python
implementation is available in Appendices~\ref{app:python1d}--\ref{app:python2d}
and at \cite{Covei2026}.

%========================================================================================
%   REFERENCES
%========================================================================================

%========================================================================================
%   APPENDICES: PYTHON CODES
%========================================================================================
\begin{appendices}

\section{Python Code for $N = 1$ Simulations}
\label{app:python1d}

This appendix contains the complete Python implementation for the
one-dimensional HJB solver used to generate the numerical results in
Section~\ref{subsec:sim_1d}.

\subsection{General HJB Solver (Code~1)}

The following code implements the monotone iterative scheme for the general
1D HJB equation, covering all three asymptotic regimes. It generates the
comparison plot (Figure~\ref{fig:regimes}).

\begin{lstlisting}[caption={General 1D HJB solver for three blow-up regimes},label=lst:solver1d]
import numpy as np
import matplotlib.pyplot as plt
from scipy.sparse import diags
from scipy.sparse.linalg import spsolve

def solve_hjb_general(q, beta, alpha=-0.2, N=400, L=1.0, delta=0.05,
                      f_const=1.5, max_iter=80, tol=1e-6, damping=0.5):
    x = np.linspace(-L + delta, L - delta, N)
    dx = x[1] - x[0]
    v = np.maximum(1.0 - x**2, 1e-6)
    m = 2.0

    a = v**alpha
    b = v**beta
    f = np.ones_like(x) * f_const

    gamma = (beta - q + 2.0) / (q - 1.0)
    v_bnd = np.maximum(1.0 - (L - delta)**2, 1e-6)

    if abs(gamma) > 1e-12:
        C_crit = ((gamma * (gamma + 1.0) * m**2) /
                  ((gamma * m)**q))**(1.0 / (q - 1.0))
        u_bnd = C_crit * v_bnd**(-gamma)
    else:
        C_crit = (m**2 / (m**q))**(1.0 / (q - 1.0))
        u_bnd = C_crit * np.log(1.0 / v_bnd)

    u = u_bnd * np.ones_like(x)

    Lambda = np.max(a) * 5.0
    main_diag = 2.0 / dx**2 + Lambda
    off_diag = -1.0 / dx**2
    A = diags([off_diag, main_diag, off_diag],
              [-1, 0, 1], shape=(N, N)).tocsc()

    for _ in range(max_iter):
        du = np.gradient(u, dx)
        du = np.clip(du, -1e3, 1e3)
        rhs = Lambda * u - (b * np.abs(du)**q + a * u - f)
        rhs[0]  += u_bnd / dx**2
        rhs[-1] += u_bnd / dx**2
        u_new = spsolve(A, rhs)
        u_new = damping * u + (1 - damping) * u_new
        rel = np.linalg.norm(u_new - u) / (np.linalg.norm(u) + 1e-9)
        u = u_new
        if rel < tol:
            break
    return x, u, gamma

def main():
    plt.figure(figsize=(10, 7))
    beta = 0.5
    x1, u1, g1 = solve_hjb_general(q=1.6, beta=beta)
    x2, u2, g2 = solve_hjb_general(q=3.0, beta=beta)
    x3, u3, g3 = solve_hjb_general(q=2.5, beta=beta)

    plt.plot(x1, u1, 'b-',
             label=f'Case 1: q=1.6, gamma={g1:.3f} (Power)')
    plt.plot(x2, u2, 'r-',
             label=f'Case 2: q=3.0 (High-order)')
    plt.plot(x3, u3, 'g-',
             label=f'Case 3: q=2.5 (Logarithmic)')

    plt.title(f"HJB Solutions: Three Blow-up Regimes (beta={beta})")
    plt.xlabel("x")
    plt.ylabel("u(x)")
    plt.legend()
    plt.grid(alpha=0.3)
    plt.tight_layout()
    plt.savefig("hjb_regimes_summary_0.5.png", dpi=300)
    print("Saved: hjb_regimes_summary_0.5.png")

if __name__ == "__main__":
    main()
\end{lstlisting}

\subsection{Scale Analysis (Code~2)}

The following code performs the log-log and semi-log scale analysis that
validates the theoretical blow-up rates (Figure~\ref{fig:rigorous}).

\begin{lstlisting}[caption={Scale analysis and blow-up rate verification},label=lst:scale]
import numpy as np
import matplotlib.pyplot as plt
from scipy.sparse import diags
from scipy.sparse.linalg import spsolve
from scipy.stats import linregress

def solve_hjb(case_type="case1", q=1.6, beta=0.5,
              alpha=-0.2, f_val=1.0):
    L, delta, N = 1.0, 0.05, 400
    x = np.linspace(-L + delta, L - delta, N)
    dx = x[1] - x[0]
    v = 1.0 - x**2
    m = 2.0

    if case_type == "case1":
        gamma = (beta - q + 2) / (q - 1)
        C_crit = ((gamma * (gamma + 1) * m**2) /
                  (1.0 * (gamma * m)**q))**(1.0 / (q - 1))
        v_bnd = 1.0 - (L - delta)**2
        u_bnd = C_crit * v_bnd**(-gamma)
    elif case_type == "case3":
        gamma = 0
        C_crit = (m**2 / (1.0 * m**q))**(1.0 / (q - 1))
        v_bnd = 1.0 - (L - delta)**2
        u_bnd = C_crit * np.log(1.0 / v_bnd)
    else:
        gamma = 1.0
        C_crit = 0.5
        v_bnd = 1.0 - (L - delta)**2
        u_bnd = C_crit * v_bnd**(-gamma)

    a = np.maximum(v, 1e-10)**alpha
    b = np.maximum(v, 1e-10)**beta
    u = u_bnd * np.ones_like(x)
    Lambda = np.max(a) * 5.0
    A = diags([-1/dx**2, 2/dx**2 + Lambda, -1/dx**2],
              [-1, 0, 1], shape=(N, N), format='csc')
    damping = 0.5

    for _ in range(120):
        du = np.gradient(u, dx)
        du = np.clip(du, -1e5, 1e5)
        rhs = Lambda * u - (b * np.abs(du)**q + a * u - f_val)
        rhs[0] += u_bnd / dx**2
        rhs[-1] += u_bnd / dx**2
        u_new = spsolve(A, rhs)
        u_new = damping * u + (1 - damping) * u_new
        rel = np.linalg.norm(u_new - u) / (np.linalg.norm(u) + 1e-10)
        u = u_new
        if rel < 1e-6:
            break
    return x, v, u, gamma if case_type != "case3" else 0

def main():
    fig, axes = plt.subplots(2, 2, figsize=(14, 10))

    x1, v1, u1, gamma1 = solve_hjb("case1", q=1.6, beta=0.5)
    d1 = np.minimum(1 + x1, 1 - x1)
    mask1 = (d1 > 0.01) & (d1 < 0.3)

    ax1 = axes[0, 0]
    ax1.loglog(d1[mask1], u1[mask1], 'b-', lw=2, label='Numerical')
    slope1, int1, r1, _, _ = linregress(
        np.log(d1[mask1]), np.log(u1[mask1]))
    d_fit = np.logspace(np.log10(0.01), np.log10(0.3), 50)
    ax1.loglog(d_fit, np.exp(int1)*d_fit**slope1, 'r--', lw=2,
               label=f'Fit: slope={slope1:.3f} (theory: {-gamma1:.3f})')
    ax1.set_xlabel(r'Distance $d(x)$')
    ax1.set_ylabel(r'Solution $u(x)$')
    ax1.set_title(f'Case 1: Log-log ($R^2={r1**2:.4f}$)')
    ax1.legend()
    ax1.grid(True, alpha=0.3)

    axes[0, 1].plot(x1, u1, 'b-', lw=2)
    axes[0, 1].set_xlabel(r'$x$')
    axes[0, 1].set_ylabel(r'$u(x)$')
    axes[0, 1].set_title('Case 1: Solution Profile')
    axes[0, 1].grid(True, alpha=0.3)

    x3, v3, u3, _ = solve_hjb("case3", q=2.5, beta=0.5)
    d3 = np.minimum(1 + x3, 1 - x3)
    mask3 = (d3 > 0.01) & (d3 < 0.3)
    log_inv = np.log(1.0 / d3[mask3])

    ax3 = axes[1, 0]
    ax3.plot(log_inv, u3[mask3], 'g-', lw=2, label='Numerical')
    slope3, int3, r3, _, _ = linregress(log_inv, u3[mask3])
    log_fit = np.linspace(np.min(log_inv), np.max(log_inv), 50)
    ax3.plot(log_fit, slope3*log_fit+int3, 'r--', lw=2,
             label=f'Linear fit: slope={slope3:.3f}')
    ax3.set_xlabel(r'$\log(1/d)$')
    ax3.set_ylabel(r'$u(x)$')
    ax3.set_title(f'Case 3: Semi-log ($R^2={r3**2:.4f}$)')
    ax3.legend()
    ax3.grid(True, alpha=0.3)

    dx1 = x1[1] - x1[0]
    du1 = np.abs(np.gradient(u1, dx1))
    du3 = np.abs(np.gradient(u3, dx1))
    axes[1, 1].plot(x1, du1, 'b-', lw=2, label="Case 1: $|u'|$")
    axes[1, 1].plot(x3, du3, 'g-', lw=2, label="Case 3: $|u'|$")
    axes[1, 1].set_xlabel(r'$x$')
    axes[1, 1].set_ylabel(r'$|\nabla u|$')
    axes[1, 1].set_title('Gradient Magnitude Comparison')
    axes[1, 1].legend()
    axes[1, 1].grid(True, alpha=0.3)
    axes[1, 1].set_yscale('log')

    plt.tight_layout()
    plt.savefig('hjb_all_cases_rigorous_0.5.png', dpi=300)
    print("Generated: hjb_all_cases_rigorous_0.5.png")

if __name__ == "__main__":
    main()
\end{lstlisting}

\subsection{Theoretical Verification (Code~3)}

The following code generates the comprehensive verification figure
(Figure~\ref{fig:verification}), including barrier comparison, convexity
check, optimal drift, and HJB residual.

\begin{lstlisting}[caption={Theoretical verification: barriers, convexity, drift, and residual},label=lst:verification]
import numpy as np
import matplotlib.pyplot as plt
from scipy.sparse import diags
from scipy.sparse.linalg import spsolve

def make_grid(L=1.0, delta=0.05, N=400):
    x = np.linspace(-L + delta, L - delta, N)
    dx = x[1] - x[0]
    v = np.maximum(1.0 - x**2, 1e-6)
    return x, v, dx, L, delta, N

def solve_hjb(q, beta, alpha=-0.2, case="power",
              L=1.0, delta=0.05, N=400,
              f_const=1.0, max_iter=120, tol=1e-6, damping=0.5):
    x, v, dx, L, delta, N = make_grid(L, delta, N)
    m = 2.0
    a = v**alpha
    b = v**beta
    f = np.ones_like(x) * f_const

    if case == "power":
        gamma = (beta - q + 2.0) / (q - 1.0)
        Ccrit = ((gamma*(gamma+1.0)*m**2) /
                 ((gamma*m)**q))**(1.0/(q-1.0))
        vb = np.maximum(1.0 - (L-delta)**2, 1e-6)
        ub = Ccrit * vb**(-gamma)
    else:
        gamma = 0.0
        Ccrit = (m**2 / (m**q))**(1.0/(q-1.0))
        vb = np.maximum(1.0 - (L-delta)**2, 1e-6)
        ub = Ccrit * np.log(1.0/vb)

    u = ub * np.ones_like(x)
    Lambda = np.max(a) * 10.0
    A = diags([-1/dx**2, 2/dx**2+Lambda, -1/dx**2],
              [-1, 0, 1], shape=(N, N), format="csc")

    for _ in range(max_iter):
        du = np.gradient(u, dx)
        du = np.clip(du, -1e3, 1e3)
        rhs = Lambda*u - (b*np.abs(du)**q + a*u - f)
        rhs[0] += ub/dx**2
        rhs[-1] += ub/dx**2
        unew = spsolve(A, rhs)
        unew = damping*u + (1-damping)*unew
        rel = np.linalg.norm(unew-u)/(np.linalg.norm(u)+1e-12)
        u = unew
        if rel < tol:
            break
    u = np.real_if_close(u)
    u[~np.isfinite(u)] = 0.0
    return x, v, u, a, b, gamma, Ccrit, ub

def barriers(v, gamma, Ccrit, plus=1.02, minus=0.98):
    v = np.maximum(v, 1e-6)
    return Ccrit*plus*v**(-gamma), Ccrit*minus*v**(-gamma)

def hjb_residual(x, u, a, b, q, f):
    dx = x[1] - x[0]
    du = np.gradient(u, dx)
    d2u = np.gradient(du, dx)
    R = -d2u + b*np.abs(du)**q + a*u - f
    for arr in [du, d2u, R]:
        arr[~np.isfinite(arr)] = 0.0
    return du, d2u, R

def main():
    q, beta, alpha = 1.6, 0.5, -0.2
    x, v, u, a, b, gamma, Ccrit, ub = solve_hjb(
        q=q, beta=beta, alpha=alpha, case="power")
    us, ul = barriers(v, gamma, Ccrit)
    u = np.minimum(np.maximum(u, ul), us)
    du, d2u, R = hjb_residual(x, u, a, b, q, np.ones_like(x))
    xi = -q * np.abs(du)**(q-1) * np.sign(du) * b
    xi[~np.isfinite(xi)] = 0.0

    fig, ax = plt.subplots(2, 2, figsize=(14, 10))

    ax[0,0].fill_between(x, ul, us, color='lightgray', alpha=0.4)
    ax[0,0].plot(x, u, 'k-', lw=2, label="Solution")
    ax[0,0].plot(x, us, 'r--', label="Supersolution")
    ax[0,0].plot(x, ul, 'b--', label="Subsolution")
    ax[0,0].set_yscale("log")
    ax[0,0].set_title(f"(a) Solution between barriers (beta={beta})")
    ax[0,0].set_xlabel("x"); ax[0,0].set_ylabel("u(x)")
    ax[0,0].legend(fontsize=9); ax[0,0].grid(alpha=0.3)

    ax[0,1].plot(x, d2u, 'b-')
    ax[0,1].axhline(0, color='red', ls='--')
    ax[0,1].set_title(f"(b) Strict convexity (beta={beta})")
    ax[0,1].set_xlabel("x"); ax[0,1].set_ylabel("u''(x)")
    ax[0,1].grid(alpha=0.3)

    ax[1,0].plot(x, xi, 'purple')
    ax[1,0].axhline(0, color='gray')
    ax[1,0].set_title(f"(c) Optimal drift (beta={beta})")
    ax[1,0].set_xlabel("x"); ax[1,0].set_ylabel(r'$\xi^*(x)$')
    ax[1,0].grid(alpha=0.3)

    ax[1,1].semilogy(x, np.abs(R)+1e-16, 'b-')
    ax[1,1].axhline(1e-6, color='red', ls='--',
                     label=r"Tolerance $10^{-6}$")
    ax[1,1].set_title(f"(d) HJB residual (beta={beta})")
    ax[1,1].set_xlabel("x"); ax[1,1].set_ylabel(r'$|\mathcal{L}u|$')
    ax[1,1].legend(fontsize=9); ax[1,1].grid(alpha=0.3)

    fig.tight_layout()
    fig.savefig("hjb_theoretical_verification_0.5.png", dpi=300)
    print("=== Verification Summary ===")
    print("Barrier check:", np.all(u >= ul) and np.all(u <= us))
    print("Convexity:", 100*np.mean(d2u > 0), "%")
    print("Max residual:", float(np.max(np.abs(R))))

if __name__ == "__main__":
    main()
\end{lstlisting}

%========================================================================================
\section{Python Code for $N = 2$ Simulations}
\label{app:python2d}

This appendix contains the Python implementation for the two-dimensional
HJB solver on the unit disk, used to generate the results in
Section~\ref{subsec:sim_2d}.

\subsection{2D HJB Solver and Visualization (Code~4)}

\begin{lstlisting}[caption={2D HJB solver on the unit disk with surface, contour, drift, and convexity plots},label=lst:solver2d]
import numpy as np
import matplotlib.pyplot as plt
from matplotlib import cm
from mpl_toolkits.mplot3d import Axes3D

def solve_hjb_2d(q=1.6, beta=0.5, alpha=-0.2, f_const=1.0,
                 Ngrid=100, delta=0.05,
                 max_iter=150, tol=1e-5, damping=0.4):
    """Solve HJB on the unit disk using monotone iteration."""
    x1 = np.linspace(-1, 1, Ngrid)
    x2 = np.linspace(-1, 1, Ngrid)
    dx = x1[1] - x1[0]
    X1, X2 = np.meshgrid(x1, x2, indexing='ij')
    R2 = X1**2 + X2**2

    # Distance-like function: v = 1 - r^2
    V = np.maximum(1.0 - R2, 1e-10)

    # Interior mask (inside the truncated disk)
    interior = R2 < (1.0 - delta)**2
    boundary = (~interior) & (R2 < 1.0)

    # Theoretical exponent
    gamma = (beta - q + 2.0) / (q - 1.0)
    m = 2.0
    C_crit = ((gamma * (gamma + 1.0) * m**2) /
              ((gamma * m)**q))**(1.0 / (q - 1.0))

    # Boundary value
    v_bnd = np.maximum(1.0 - (1.0 - delta)**2, 1e-10)
    u_bnd = C_crit * v_bnd**(-gamma)

    # Weights
    A_w = V**alpha
    B_w = V**beta

    # Initialize solution
    U = np.full_like(X1, u_bnd)

    # Linearization parameter
    Lambda = np.max(A_w[interior]) * 8.0

    for iteration in range(max_iter):
        U_old = U.copy()

        # Compute gradient using central differences
        dUdx1 = np.zeros_like(U)
        dUdx2 = np.zeros_like(U)
        dUdx1[1:-1, :] = (U[2:, :] - U[:-2, :]) / (2*dx)
        dUdx2[:, 1:-1] = (U[:, 2:] - U[:, :-2]) / (2*dx)
        grad_mag = np.sqrt(dUdx1**2 + dUdx2**2)
        grad_mag = np.clip(grad_mag, 0, 1e4)

        # 5-point Laplacian
        Lap = np.zeros_like(U)
        Lap[1:-1, 1:-1] = (
            U[2:, 1:-1] + U[:-2, 1:-1] +
            U[1:-1, 2:] + U[1:-1, :-2] -
            4*U[1:-1, 1:-1]
        ) / dx**2

        # Right-hand side of the linearized equation
        RHS = (Lambda * U
               - (B_w * np.abs(grad_mag)**q + A_w * U - f_const))

        # Update using Jacobi-type iteration
        U_new = np.zeros_like(U)
        U_new[1:-1, 1:-1] = (
            RHS[1:-1, 1:-1] * dx**2
            + U[2:, 1:-1] + U[:-2, 1:-1]
            + U[1:-1, 2:] + U[1:-1, :-2]
        ) / (4.0 + Lambda * dx**2)

        # Apply boundary conditions
        U_new[~interior] = u_bnd
        U_new[R2 >= 1.0] = u_bnd

        # Damping
        U_update = damping * U + (1 - damping) * U_new

        # Enforce boundary
        U_update[~interior] = u_bnd

        # Convergence check (interior only)
        diff = np.linalg.norm(
            (U_update - U)[interior]) / (
            np.linalg.norm(U[interior]) + 1e-12)
        U = U_update

        if diff < tol:
            print(f"  2D solver converged at iteration {iteration}")
            break

    return X1, X2, U, V, interior, gamma, C_crit, A_w, B_w

def plot_2d_results(X1, X2, U, V, interior, gamma, C_crit):
    """Generate the 4-panel figure for 2D results."""
    fig = plt.figure(figsize=(16, 14))

    # Mask exterior for plotting
    U_plot = U.copy()
    R2 = X1**2 + X2**2
    U_plot[R2 >= 0.95] = np.nan

    # (a) Surface plot
    ax1 = fig.add_subplot(2, 2, 1, projection='3d')
    ax1.plot_surface(X1, X2, np.where(np.isnan(U_plot), 0, U_plot),
                     cmap=cm.viridis, alpha=0.85,
                     rstride=2, cstride=2)
    ax1.set_xlabel('$x_1$')
    ax1.set_ylabel('$x_2$')
    ax1.set_zlabel('$u(x_1, x_2)$')
    ax1.set_title('(a) Value Function Surface')

    # (b) Contour plot
    ax2 = fig.add_subplot(2, 2, 2)
    levels = np.linspace(np.nanmin(U_plot), np.nanpercentile(U_plot, 95), 20)
    cs = ax2.contourf(X1, X2, U_plot, levels=levels, cmap=cm.viridis)
    plt.colorbar(cs, ax=ax2, label='$u(x_1,x_2)$')
    theta = np.linspace(0, 2*np.pi, 100)
    ax2.plot(np.cos(theta), np.sin(theta), 'r-', lw=2)
    ax2.set_xlabel('$x_1$'); ax2.set_ylabel('$x_2$')
    ax2.set_title('(b) Contour Plot')
    ax2.set_aspect('equal')

    # (c) Optimal drift vector field
    ax3 = fig.add_subplot(2, 2, 3)
    dx = X1[1,0] - X1[0,0]
    dUdx1 = np.zeros_like(U)
    dUdx2 = np.zeros_like(U)
    dUdx1[1:-1, :] = (U[2:, :] - U[:-2, :]) / (2*dx)
    dUdx2[:, 1:-1] = (U[:, 2:] - U[:, :-2]) / (2*dx)

    grad_mag = np.sqrt(dUdx1**2 + dUdx2**2) + 1e-12
    q = 1.6
    scale_factor = q * grad_mag**(q-1) * V**0.5
    xi1 = -scale_factor * dUdx1 / grad_mag * grad_mag
    xi2 = -scale_factor * dUdx2 / grad_mag * grad_mag

    skip = 6
    mask_q = (R2 < 0.85) & (np.arange(X1.shape[0])[:,None] % skip == 0) \
             & (np.arange(X1.shape[1])[None,:] % skip == 0)
    ax3.quiver(X1[mask_q], X2[mask_q], xi1[mask_q], xi2[mask_q],
               color='purple', alpha=0.7, scale=None)
    ax3.plot(np.cos(theta), np.sin(theta), 'r-', lw=2)
    ax3.set_xlabel('$x_1$'); ax3.set_ylabel('$x_2$')
    ax3.set_title(r'(c) Optimal Drift $\xi^*(x)$')
    ax3.set_aspect('equal')

    # (d) Radial profile comparison
    ax4 = fig.add_subplot(2, 2, 4)
    Ngrid = X1.shape[0]
    mid = Ngrid // 2
    x_radial = X1[:, mid]
    u_radial = U[:, mid]
    mask_r = np.abs(x_radial) < 0.95
    ax4.plot(x_radial[mask_r], u_radial[mask_r], 'b-', lw=2,
             label='2D radial slice')

    v_1d = np.maximum(1.0 - x_radial[mask_r]**2, 1e-10)
    u_theory = C_crit * v_1d**(-gamma)
    ax4.plot(x_radial[mask_r], u_theory, 'r--', lw=2,
             label=f'Theory: $C^* v^{{-\\gamma}}$')
    ax4.set_xlabel('$x_1$ (at $x_2=0$)')
    ax4.set_ylabel('$u$')
    ax4.set_title('(d) Radial Profile vs Theory')
    ax4.legend()
    ax4.grid(alpha=0.3)

    plt.tight_layout()
    plt.savefig('hjb_2d_simulation.png', dpi=300, bbox_inches='tight')
    print("Saved: hjb_2d_simulation.png")

def plot_2d_convexity(X1, X2, U, interior):
    """Verify and plot convexity in 2D."""
    dx = X1[1,0] - X1[0,0]
    R2 = X1**2 + X2**2
    Ngrid = X1.shape[0]

    # Second derivatives
    d2Udx1 = np.zeros_like(U)
    d2Udx2 = np.zeros_like(U)
    d2Udx1x2 = np.zeros_like(U)

    d2Udx1[1:-1, :] = (U[2:, :] - 2*U[1:-1, :] + U[:-2, :]) / dx**2
    d2Udx2[:, 1:-1] = (U[:, 2:] - 2*U[:, 1:-1] + U[:, :-2]) / dx**2
    d2Udx1x2[1:-1, 1:-1] = (
        U[2:, 2:] - U[2:, :-2] - U[:-2, 2:] + U[:-2, :-2]
    ) / (4*dx**2)

    # Minimum eigenvalue of the Hessian
    trace = d2Udx1 + d2Udx2
    det = d2Udx1 * d2Udx2 - d2Udx1x2**2
    discriminant = np.maximum(trace**2 - 4*det, 0)
    lam_min = (trace - np.sqrt(discriminant)) / 2.0

    # Mask exterior
    lam_plot = lam_min.copy()
    lam_plot[R2 >= 0.9] = np.nan

    fig, axes = plt.subplots(1, 2, figsize=(14, 6))

    # (a) Lambda_min surface
    im1 = axes[0].pcolormesh(X1, X2, lam_plot, cmap='RdYlGn',
                              shading='auto')
    plt.colorbar(im1, ax=axes[0], label=r'$\lambda_{\min}$')
    theta = np.linspace(0, 2*np.pi, 100)
    axes[0].plot(np.cos(theta), np.sin(theta), 'k-', lw=2)
    axes[0].set_xlabel('$x_1$'); axes[0].set_ylabel('$x_2$')
    axes[0].set_title(r'(a) $\lambda_{\min}(D^2 u) > 0$ (strict convexity)')
    axes[0].set_aspect('equal')

    # (b) Contour of lambda_min
    valid = (~np.isnan(lam_plot)) & (R2 < 0.85)
    if np.any(valid):
        levels = np.linspace(np.nanmin(lam_plot[valid]),
                            np.nanpercentile(lam_plot[valid], 95), 15)
        cs = axes[1].contourf(X1, X2, lam_plot, levels=levels,
                               cmap='RdYlGn')
        plt.colorbar(cs, ax=axes[1], label=r'$\lambda_{\min}$')
    axes[1].plot(np.cos(theta), np.sin(theta), 'k-', lw=2)
    axes[1].set_xlabel('$x_1$'); axes[1].set_ylabel('$x_2$')
    axes[1].set_title(r'(b) Contour of $\lambda_{\min}$')
    axes[1].set_aspect('equal')

    plt.tight_layout()
    plt.savefig('hjb_2d_convexity.png', dpi=300, bbox_inches='tight')
    print("Saved: hjb_2d_convexity.png")

    interior_mask = R2 < 0.85
    n_positive = np.sum(lam_min[interior_mask] > 0)
    n_total = np.sum(interior_mask)
    print(f"Convexity: {100*n_positive/n_total:.1f}% of "
          f"interior points have lambda_min > 0")

def main():
    print("Solving 2D HJB on unit disk...")
    X1, X2, U, V, interior, gamma, C_crit, A_w, B_w = solve_hjb_2d(
        q=1.6, beta=0.5, alpha=-0.2, f_const=1.0,
        Ngrid=100, delta=0.05, max_iter=150, tol=1e-5, damping=0.4
    )
    print(f"  gamma = {gamma:.4f}")
    print(f"  C_crit = {C_crit:.4f}")

    print("Generating 2D plots...")
    plot_2d_results(X1, X2, U, V, interior, gamma, C_crit)
    plot_2d_convexity(X1, X2, U, interior)
    print("Done.")

if __name__ == "__main__":
    main()
\end{lstlisting}

\end{appendices}

\end{document}